\documentclass{article}
\usepackage{bbm}

\RequirePackage{amsmath,amssymb,amsthm,mathtools}
\usepackage[T1]{fontenc}
\usepackage{newtxtext}
\usepackage{newtxmath}

\RequirePackage{geometry}
\RequirePackage{microtype}
\RequirePackage{graphicx}
\RequirePackage{xcolor}
\RequirePackage{titlesec}
\RequirePackage{fancyhdr}
\RequirePackage{enumitem}
\RequirePackage{framed}
\RequirePackage{hyperref}
\RequirePackage{cleveref}

\definecolor{Accent}{HTML}{0D9488}   
\definecolor{Muted}{HTML}{6B7280}   
\definecolor{Rule}{HTML}{E5E7EB}

\hypersetup{
  colorlinks=true,
  linkcolor=Accent,
  citecolor=Accent,
  urlcolor=Accent,
  pdfborder={0 0 0}
}

\usepackage{caption}
\usepackage{bbm}
\usepackage{tikz}
\usetikzlibrary{arrows.meta, calc}
\usepackage{geometry}
\numberwithin{equation}{section}

\theoremstyle{plain}
\newtheorem{theorem}{Theorem}[section]
\newtheorem{lemma}[theorem]{Lemma}
\newtheorem{proposition}[theorem]{Proposition}
\newtheorem{corollary}[theorem]{Corollary}

\theoremstyle{definition}
\newtheorem{definition}[theorem]{Definition}
\newtheorem{example}[theorem]{Example}

\theoremstyle{remark}
\newtheorem*{remark}{Remark}

\setlist{topsep=0.35em,itemsep=0.2em,parsep=0pt}

\newcommand{\RR}{\mathbf{R}}
\newcommand{\CC}{\mathbf{C}}
\newcommand{\NatNum}{\mathbf{N}}

\newcommand{\di}[1]{\mathop{\mathrm{d}#1}}

\newcommand{\conj}[1]{\overline{#1}}
\newcommand{\cl}[1]{\mathrm{cl}({#1})}

\newcommand{\mom}{\mathfrak{m}}

\newcommand{\MM}{\mathcal{M}}
\newcommand{\NN}{\mathcal{N}}
\newcommand{\PP}{\mathcal{P}}

\newcommand{\ee}{\varepsilon}

\newcommand{\spec}{\mathrm{spec}}

\newcommand{\stard}{\stackrel{\mathrm{*d}}{\sim}}

\newcommand{\supp}{\mathrm{supp}}

\DeclarePairedDelimiterX{\inner}[2]{\langle}{\rangle}{#1,#2}

\newcommand{\BBP}{\mathbb{B}}
\newcommand{\iu}{\mathrm{i}}  

\newcommand{\Rsf}{\mathsf{R}}
\newcommand{\Gsf}{\mathsf{G}}

\usepackage{enumitem}
\setenumerate[1]{label={\textnormal{(\arabic*)}}, itemsep=0.2em,topsep=0.2em}
\setlist{nolistsep}

\usepackage[
    backend=biber,
    sorting=nyt,
    style=alphabetic,
    url=false,
    isbn=false,
    giveninits=true,
    date=year,
    backref=true,
      maxnames=5
]{biblatex}
\renewbibmacro{in:}{
  \ifboolexpr{
     test {\ifentrytype{article}}
     or
     test {\ifentrytype{inproceedings}}
  }{}{\printtext{\bibstring{in}\intitlepunct}}
}

\usepackage{etoolbox}

\makeatletter
\patchcmd{\@maketitle}
  {\LARGE \@title \par}
  {\fontsize{16}{19}\selectfont \@title \par}
  {}{}
\makeatother

\title{\textbf{Freely infinitely divisible $R$-diagonal elements and Brown measure}}
\author{
  Yu Kitagawa\thanks{Department of Mathematics, Faculty of Science, Kyoto University, Japan. Email: \texttt{kitagawa.yu.57z@st.kyoto-u.ac.jp}}
  \and
  Mihai Popa\thanks{Department of Mathematics, The University of Texas at San Antonio, San Antonio, TX 78249, USA. Email: \texttt{mihai.popa@utsa.edu}}
  \and
  Ping Zhong\thanks{Department of Mathematics, University of Houston, Houston, TX 77204-3008, USA. Email: \texttt{pzhong@central.uh.edu}}
}

\date{\today}

\begin{document}

\maketitle

\begin{abstract}
We study freely infinitely divisible $R$-diagonal elements in the unbounded setting and Brown measures for free additive perturbations by such elements. This class includes circular elements, circular Cauchy elements, and other previously studied $R$-diagonal models. We construct examples and prove stability under several algebraic operations, including homogeneous noncommutative polynomials in bounded, freely independent elements from this class. Using results for general $R$-diagonal perturbations, together with several analytic estimates specific to freely infinitely divisible $R$-diagonal elements, we prove that, in the bounded case, the support of the Brown measure coincides with the spectrum, and we obtain a criterion for property (H) in this non-normal setting. Finally, we study the free convolution semigroup associated with the symmetrized law of the modulus and derive a Hamilton--Jacobi equation for the regularized logarithmic potential.
\end{abstract}

\tableofcontents

\section{Introduction}

The Brown measure, introduced in \cite{B1986lidskii}, is the natural analogue of the spectral distribution for non-normal operators in a tracial von Neumann algebra. Unlike in the self-adjoint case, however, Brown measures are usually difficult to compute and even harder to analyze. $R$-diagonal elements form one of the few classes of non-normal operators for which the Brown measure admits an explicit description. They were introduced in \cite{NS1997rdiag} to provide a unified framework for treating two fundamental examples of non-normal operators: circular elements and Haar unitary elements. In the bounded setting, it was shown in \cite{HL2000Brown} that the Brown measure of an $R$-diagonal element is rotationally invariant and completely determined by the law of its modulus; this description was later extended to the unbounded setting in \cite{HS2007Brown}. See also \cite{NicaSpeicher2006book,Z2023brown} for background and related reformulations. This explicit description makes $R$-diagonal elements a natural testing ground for regularity problems for Brown measures in non-normal free probability.

Several developments suggest that the behavior of Brown measures under addition becomes more tractable when one summand has additional structure. The circular element is the basic example of such a structured perturbation. The Brown measures of free circular Brownian motion have been studied extensively, and a variety of regularity results have been established in \cite{BL2001computation,BCC2014spectrum,BC2016outlier,HZ2023Brown,BYZ2024Brown,EJ2025density,Z2026brown}. Beyond the circular case, the circular Cauchy operator appears in \cite{HS2007Brown} as an unbounded $R$-diagonal example, and this operator also plays an important role in the argument for the existence of invariant subspaces in \cite{HS2009inv}. More generally, Brown measures of sums with one $R$-diagonal summand were analyzed in \cite{BZ2022}. Such perturbations are also fundamental in random matrix theory, as shown by the single-ring and deformed single-ring theorems \cite{GKZ2011single,HZ2025deformed}. These results suggest that $R$-diagonal perturbations exhibit a strong regularizing effect.

In the self-adjoint setting, the regularizing effects of free convolution have been studied extensively, beginning with free convolution with the semicircular distribution \cite{B1997free} and later extending to the framework of freely infinitely divisible distributions \cite{HW2022Reg,BWZ2023superconvergence}. This line of work has produced powerful analytic tools based on the Voiculescu transform and subordination functions. These methods make it possible to describe the support, atoms, density, and related regularity properties of freely infinitely divisible laws.

The aim of this paper is to develop a non-normal counterpart of the regularity theory for freely infinitely divisible laws in the $R$-diagonal setting. The class of freely infinitely divisible $R$-diagonal elements was introduced in \cite{BNNS2018eta} in a not necessarily tracial setting, using a cumulant-based argument. A related viewpoint has also been considered in the framework of rectangular free convolution and its infinitely divisible distributions \cite{Bg2007infinitely,Bg2009rectangular}. This class contains several canonical examples, including circular elements, circular Cauchy elements, and other $R$-diagonal models appearing in the random matrix and random polynomial literature \cite{KT2018limiting,COR2024fractional}.

A fundamental fact about $R$-diagonal elements is that their $*$-distribution is governed by the distribution of the modulus, and the symmetrized law of the modulus is compatible with free additive convolution. This allows results on symmetric freely infinitely divisible distributions to be transferred to the $R$-diagonal setting. In particular, we obtain the following characterization, as well as a stability result for homogeneous noncommutative polynomials in the bounded setting; see Theorems~\ref{thm:rid_char} and~\ref{thm:prod} for the precise statements.

\begin{theorem}
Let $X$ be an $R$-diagonal element. Then $X$ is freely infinitely divisible if and only if the symmetrized law $\widetilde{\mu}_{|X|}$ is freely infinitely divisible. In the bounded case, any homogeneous noncommutative polynomial in freely independent elements from this class is again $R$-diagonal and freely infinitely divisible.
\end{theorem}

We next turn to Brown measures for freely infinitely divisible $R$-diagonal elements and for additive perturbations by them. Building on the analytic framework of \cite{BZ2022}, we prove results on the support of the Brown measure, including the coincidence of the support with the spectrum in the bounded case, as well as a criterion for property (H), a regularity property introduced and studied in \cite{BBgG2009regularization,HW2022Reg}. See Proposition~\ref{prop:l-2_infty}, Corollary~\ref{cor:supp_spec}, and Theorem~\ref{thm:property_H} for the precise statements.

\begin{theorem}
Let $Y\in\log^+(\MM)$ be a freely infinitely divisible $R$-diagonal element. Then $\mom_{-2}(Y)=\infty$. In particular, if $Y$ is bounded, then $\supp(\mu_Y)=\spec(Y)$. Moreover, $Y$ satisfies property (H), that is, for every $X_0\in\log^+(\MM)$ that is $*$-free from $Y$, the Brown measure of $X_0+Y$ is absolutely continuous and has strictly positive density on $\CC$, if and only if $\tau(\ker Y)=0$ and $\mom_2(Y)=\infty$.
\end{theorem}

Finally, we study the free convolution semigroup associated with the symmetrized law of the modulus. This semigroup provides a dynamical framework for studying regularity properties of Brown measures. It includes the free circular Brownian motion setting \cite{HZ2023Brown,BYZ2024Brown} and the circular Cauchy model \cite{HS2007Brown,HS2009inv}. In this setting, for perturbations $X_t=X_0+Y_t$ satisfying $\widetilde{\mu}_{|Y_t|}=\mu^{\boxplus t}$, where $\widetilde{\mu}_{|Y_t|}$ denotes the symmetrization of the law of $|Y_t|$, we derive a Hamilton--Jacobi equation for the regularized logarithmic potential. See Theorem~\ref{thm:determinantal-regularity} and Proposition~\ref{prop:hj_pde} for the precise statements.

\begin{theorem}
Let $S(t,\lambda,\ee)=\tau(\log(|X_t-\lambda|^2+\ee^2))$. Then $S$ satisfies a Hamilton--Jacobi equation whose Hamiltonian is determined by the $R$-transform of $\mu$.
\end{theorem}

The main analytic inputs of the paper come from two sources. The first is the global inversion formula for freely infinitely divisible laws developed in \cite{B1997free, BB2004, BB2005partially, HW2022Reg}. The second is the subordination approach to Brown measures of sums with one $R$-diagonal summand developed in \cite{BZ2022}; see also \cite{BBS2018eigenvalues} for an earlier and more general approach.

\subsection*{Organization of the paper}

Section~\ref{sec:preliminaries} recalls the necessary background. Sections~\ref{sec:id-rdiagonal} and~\ref{sec:homogeneous-polynomials} study the structure and several properties of freely infinitely divisible $R$-diagonal elements. Section~\ref{sec:support_infinitely_divisible_rdiag} studies the support of Brown measures and property (H) for freely infinitely divisible $R$-diagonal elements, and Section~\ref{sec:semigroup-regularity} treats the associated semigroup and derives a Hamilton--Jacobi equation and other PDE descriptions.

\section{Preliminaries}
\label{sec:preliminaries}

Throughout this paper, we work in a tracial $W^*$-probability space $(\MM,\tau)$ unless stated otherwise. 
Let $\widetilde{\MM}$ denote the collection of all closed densely defined operators affiliated with $\MM$. For a self-adjoint operator $T\in\widetilde{\MM}$, let $\mu_{T}$ denote the spectral measure of $T$. 

For $A\in\widetilde{\MM}$ and $p\in\RR$, we write
\begin{align*}
\mom_p(A)\coloneqq \int_0^\infty t^p \di{\mu_{|A|}}(t),
\end{align*}
with the convention that the value is allowed to be $+\infty$. For negative $p$, the integrand is interpreted as $+\infty$ at $t=0$. Thus $\mom_p(A)=+\infty$ for $p<0$ whenever $\tau(\ker A)>0$, where $\ker A$ denotes the projection onto the kernel of $A$, so that $\tau(\ker A)=\mu_{|A|}(\{0\})$. In what follows, we shall mainly use the cases $p=2$ and $p=-2$.

In this section, we recall several analytic results from the theory of free additive convolution for self-adjoint operators, together with the notions of Brown measure and $R$-diagonality used later in the paper.

\subsection{Subordination functions and free infinite divisibility}

For a self-adjoint affiliated operator $Y\in\widetilde{\MM}$, we define the Cauchy transform $G_Y$ and the reciprocal Cauchy transform $F_Y$ by
\begin{align*}
G_Y(z)\coloneqq\int_{\RR}\frac{1}{z-t}\di{\mu_Y}(t),\quad F_Y(z)\coloneqq\frac{1}{G_Y(z)},
\end{align*}
for $z\in\CC^+$. The following theorem states the existence of subordination functions for the free additive convolution.

\begin{theorem}[\cite{B1998process, V1993analogue, BB2007new}]
\label{thm:subord_func}
Let $X_1,X_2\in\widetilde{\MM}$ be self-adjoint and freely independent, and assume that neither $X_1$ nor $X_2$ is a scalar multiple of the identity. Let $X\coloneqq X_1+X_2$. Then there exist unique analytic functions $\omega_1,\omega_2\colon\CC^+\to\CC^+$ such that, for every $z\in\CC^+$,
\begin{enumerate}
\item $F_X(z)=F_{X_1}(\omega_1(z))=F_{X_2}(\omega_2(z))$,
\item $F_X(z)+z=\omega_1(z)+\omega_2(z)$, and
\item $\lim_{y\to\infty}\frac{\omega_1(\iu y)}{\iu y}=\lim_{y\to\infty}\frac{\omega_2(\iu y)}{\iu y}=1$.
\end{enumerate}
\end{theorem}

For a self-adjoint affiliated operator $Y$, the function $F_Y$ admits an analytic right inverse $F_Y^{\langle-1\rangle}$ on a truncated cone of the form $\Gamma_{\alpha,\beta}=\{ x+\iu y\colon y>\beta, \alpha y>|x|\}$ for some $\alpha>0, \beta>0$, and the Voiculescu transform of $Y$ is defined there by
\begin{align*}
\phi_Y(z)\coloneqq F_Y^{\langle-1\rangle}(z)-z.
\end{align*}
The $R$-transform $R_Y$ of $Y$ is defined by $R_Y(w)=\phi_Y(1/w)$ in the domain $\Delta_{\alpha,\beta}\coloneqq\{z^{-1}\colon z\in\Gamma_{\alpha,\beta}\}$ in $\CC^-$. The Voiculescu transform linearizes free additive convolution; that is,
\begin{align*}
\phi_X(z)=\phi_{X_1}(z)+\phi_{X_2}(z)
\end{align*}
on the intersection of the domains, where $X_1,X_2\in \widetilde{\MM}$ are free and $X\coloneqq X_1+X_2$. Equivalently, $R_X(z)=R_{X_1}(z)+R_{X_2}(z)$ in a domain in $\CC^-$ where all these functions are defined.

The next proposition gives the standard characterization of free infinite divisibility in terms of the Voiculescu transform.

\begin{proposition}[{\cite[Theorem 5.10]{BV1993free}}]
\label{prop:free_LH}
Let $Y\in\widetilde{\MM}$ be self-adjoint. Then $Y$ is freely infinitely divisible if and only if $\phi_Y$ has an analytic continuation to $\CC^+$ with values in $\CC^-\cup\RR$. In this case there exist a unique constant $\gamma_Y\in\RR$ and a unique finite positive Borel measure $\sigma_Y$ on $\RR$ such that
\begin{align*}
\phi_Y(z)=\gamma_Y+\int_{\RR}\frac{1+sz}{z-s}\di{\sigma_Y}(s),\quad z\in\CC^+.
\end{align*}
We call $(\gamma_Y,\sigma_Y)$ the free generating pair of $Y$.
\end{proposition}

When one of the summands is freely infinitely divisible, the subordination functions admit further analytic properties. We shall use the following form of this result.

\begin{theorem}[{\cite[Theorem 4.6, Proposition 4.7]{BB2005partially}, \cite[Lemma 4]{B1997free}, \cite[Proposition 3.1]{HW2022Reg}}]
\label{thm:global_inversion}
Let $X_1,X_2\in\widetilde{\MM}$ be self-adjoint and freely independent, and assume that $X_2$ is freely infinitely divisible and not a scalar multiple of the identity. Let $X\coloneqq X_1+X_2$ and define
\begin{align*}
H(z)&\coloneqq z+\phi_{X_2}(F_{X_1}(z))=z+R_{X_2}(G_{X_1}(z)),
\end{align*}
and
\begin{align*}
\Omega_{H}\coloneqq\{z\in\CC^+ \colon \Im H(z)>0\}.
\end{align*}
Then 
\begin{enumerate}
\item $\Omega_{H}$ is simply connected.
\item The map $H\colon\Omega_{H}\to\CC^+$ is an analytic bijection and extends as a continuous map from $\cl{\Omega_{H}}$ onto $\CC^+\cup\RR$.
\item The subordination function $\omega_1\colon\CC^+\to\CC^+$ admits a continuous extension $\omega\colon\CC^+\cup\RR\to\cl{\Omega_{H}}$ such that
\begin{align*}
H(\omega(z))&=z, \quad z\in\CC^+\cup\RR,\\
\omega(H(z))&=z, \quad z\in\cl{\Omega_{H}}.
\end{align*}
Moreover, for any $z\in\CC^+$,
\begin{align*}
    \omega_1(z)=z-\phi_{X_2}(F_X(z))=z-R_{X_2}(G_X(z)).
\end{align*}
\item The boundary satisfies $\partial\Omega_H=\omega(\RR)$ and is the graph of a function on $\RR$, described explicitly in terms of $H$.
\end{enumerate}
\end{theorem}

We also recall the $S$-transform. Let $\mu$ be a probability measure on $[0,\infty)$ with $\mu(\{0\})<1$. For $z<0$, define
\begin{align*}
\psi_\mu(z)=\int_0^\infty\frac{zx}{1-zx}\di{\mu}(x).
\end{align*}
Then $\psi_\mu$ is an increasing function from $(-\infty,0)$ onto $(\mu(\{0\})-1,0)$. If $\chi_\mu\colon(\mu(\{0\})-1,0)\to(-\infty,0)$ denotes its inverse, then the $S$-transform of $\mu$ is defined by
\begin{align*}
S_\mu(u)=\frac{1+u}{u}\chi_\mu(u),\quad \mu(\{0\})-1<u<0.
\end{align*}
The $S$-transform satisfies the multiplicative property $S_{\mu\boxtimes\nu}(u)=S_\mu(u)S_\nu(u)$ for $u$ in the common domain where they are defined.

We shall also use the following well-known observation, which relates the boundary behavior of the $S$-transform to the negative first moment.
\begin{lemma}
\label{lem:s_transform_rel}
For a probability measure $\mu$ on $[0, \infty)$ with $\mu(\{0\})=0$, we have
\begin{align*}
\lim_{u\to -1}S_\mu(u)=\int_0^\infty\frac{1}{x}\di{\mu}(x)\in(0,\infty].
\end{align*}
\end{lemma}

\begin{proof}
First, by the assumption that $\mu(\{0\})=0$, we observe that $\lim_{t\to\infty}\psi_\mu(-t)=-1$ and 
\begin{align*}
\lim_{t\to\infty}t(1+\psi_\mu(-t))=\lim_{t\to\infty}\int_0^\infty\frac{1}{x+1/t}\di{\mu}(x)=\int_0^\infty\frac{1}{x}\di{\mu}(x)\in(0,\infty].
\end{align*}
Therefore, by the definition of the $S$-transform, we have
\begin{align*}
\lim_{u\to-1}S_\mu(u)=\lim_{t\to\infty}S_\mu(\psi_\mu(-t))=\lim_{t\to\infty}\frac{-t(1+\psi_\mu(-t))}{\psi_\mu(-t)}=\lim_{t\to\infty}\frac{-1}{\psi_\mu(-t)}\int_0^\infty\frac{1}{x}\di{\mu}(x)=\int_0^\infty\frac{1}{x}\di{\mu}(x).
\end{align*}
\end{proof}

\begin{remark}
The analytic functions introduced in this section, initially defined on $\CC^+$, can be extended analytically to $\CC^-$ via $f(z)\coloneqq\conj{f(\conj{z})}$. We use this fact implicitly in what follows.
\end{remark}

\subsection{$R$-diagonal elements, Brown measure, and their basic properties}

One of the central classes of operators in this paper is the class of $R$-diagonal elements. 
Recall that an operator $a\in\MM$ is $R$-diagonal if, for every $n\in\NatNum$, the cumulant $\kappa_n(a_1,a_2,\dots,a_n)$ vanishes whenever $n$ is odd or the arguments $a_1,a_2,\dots,a_n\in\{a,a^*\}$ are not alternating in $a$ and $a^*$. In the tracial case, the two alternating even cumulants coincide, and the sequence
\begin{align*}
\left(\kappa_{2n}(a,a^*,a,a^*,\dots,a,a^*)\right)_{n\ge1}=\left(\kappa_{2n}(a^*,a,a^*,a,\dots,a^*,a)\right)_{n\ge1}
\end{align*}
is called the determining sequence of $a$.
$R$-diagonal elements were introduced in \cite{NS1997rdiag}. They form a broad class of non-selfadjoint operators, including circular elements and Haar unitaries. The definition was extended to the unbounded setting in \cite{HS2007Brown}.
\begin{definition}[{\cite[Definition 3.2]{HS2007Brown}}]
Let $X_1,X_2\in\widetilde{\MM}$.
\begin{enumerate}
\item We write $X_1\stard X_2$ and say that $X_1$ and $X_2$ have the same $*$-distribution if there is a trace-preserving $*$-isomorphism $\phi\colon W^{*}(X_1)\to W^{*}(X_2)$ such that $\widetilde{\phi}(X_1)=X_2$.
\item We call $X_1$ and $X_2$ $*$-free if $W^*(X_1)$ and $W^*(X_2)$ are $*$-free.
\end{enumerate}
\end{definition}

\begin{definition}[{\cite[Definition 3.3]{HS2007Brown}}]
An operator $X\in\widetilde{\MM}$ is called $R$-diagonal if there exist a tracial $W^*$-probability space $(\NN,\tau_\NN)$ and $*$-free operators $U,H\in\widetilde{\NN}$ such that
\begin{enumerate}
\item $U$ is a Haar unitary,
\item $H\ge 0$, and
\item $X\stard UH$.
\end{enumerate}
\end{definition}

We recall the definition of Brown measure. We define
\begin{align*}
\log^+(\MM)\coloneqq\left\{T\in\widetilde{\MM}\colon\tau(\log^+|T|)=\int_0^\infty \log^+(t)\di{\mu_{|T|}}(t)<\infty\right\},
\end{align*}
where $\log^+(t)=\max\{\log t,0\}$. This is the natural class on which the Fuglede--Kadison determinant and the Brown measure are defined for possibly unbounded operators.

\begin{definition}
Let $X\in\log^+(\MM)$. The Fuglede--Kadison determinant of $X$ is defined by
\begin{align*}
\Delta(X)\coloneqq\exp(\tau(\log|X|)),
\end{align*}
where the right-hand side is understood in $[0,\infty)$. For $X\in\log^+(\MM)$, the function
\begin{align*}
L_X(\lambda)\coloneqq\log\Delta(X-\lambda)=\tau(\log|X-\lambda|), \quad \lambda\in\CC,
\end{align*}
is subharmonic on $\CC$, and the Brown measure of $X$ is the unique probability measure $\mu_X$ on $\CC$ such that
\begin{align*}
\di{\mu_X}(\lambda)=\frac{1}{2\pi}\Delta_\lambda L_X(\lambda)
\end{align*}
in the sense of distributions.
\end{definition}

For later use, we also introduce the regularized logarithmic potential
\begin{align*}
S_X(\lambda,\varepsilon)\coloneqq\tau(\log(|X-\lambda|^2+\varepsilon^2)),\quad \lambda\in\CC,\ \varepsilon>0.
\end{align*}

If $X$ is normal, then $\mu_X$ coincides with the usual spectral distribution of $X$, allowing us to use this notation consistently even for normal operators.

The Brown measure of bounded $R$-diagonal elements was studied in detail in \cite{HL2000Brown}, and this analysis was extended to the unbounded setting in \cite{HS2007Brown}. This description was also reformulated and rederived from the viewpoint of subordination in \cite{Z2023brown}. In general, for an $R$-diagonal element $T\in\log^+(\MM)$, the Brown measure $\mu_T$ is rotationally invariant and is determined by the distribution of $|T|$. Moreover, the support of the Brown measure need not coincide with the spectrum. The next theorem gives the formula for the radial distribution function and the criterion for the support to coincide with the spectrum.

\begin{theorem}[{\cite[Proposition 4.6]{HL2000Brown}, \cite[Theorem 4.17 and Remark 4.18]{HS2007Brown}}]
\label{thm:char_r_supp_spec}
Let $X\in\log^+(\MM)$ be an $R$-diagonal element such that $\mu_{|X|^2}$ is not a Dirac measure. Then $\mu_X$ is rotationally invariant and $\mu_X(\{0\})=\mu_{|X|^2}(\{0\})$. Moreover, for $r>0$,
\begin{align*}
\mu_X(\{z\in\CC \colon |z|\le r\})=
\begin{cases}
0,&0<r\le \mom_{-2}(X)^{-1/2},\\
1+S_{\mu_{|X|^2}}^{\langle-1\rangle}(r^{-2}),&\mom_{-2}(X)^{-1/2}<r<\mom_2(X)^{1/2},\\
1,&\mom_2(X)^{1/2}\le r,
\end{cases}
\end{align*}
where $S_{\mu_{|X|^2}}^{\langle-1\rangle}$ denotes the inverse function of the strictly decreasing function $S_{\mu_{|X|^2}}\colon(\mu_{|X|^2}(\{0\})-1,0)\to(\mom_2(X)^{-1},\mom_{-2}(X))$. Consequently,
\begin{align*}
\supp(\mu_X)=\{\lambda\in\CC\colon \mom_{-2}(X)^{-1/2}\le |\lambda|\le \mom_2(X)^{1/2}\}.
\end{align*}
Moreover, we have $\supp(\mu_X)\subsetneq\spec(X)$ if and only if at least one of the following holds:
\begin{enumerate}
\item $X$ does not have a bounded inverse in $\MM$ and $\mom_{-2}(X)<\infty$,
\item $X$ is unbounded and $\mom_2(X)<\infty$.
\end{enumerate}
\end{theorem}

For a probability measure $\mu$ on $\RR$, we define its symmetrization by
\begin{align*}
\widetilde{\mu}(B)\coloneqq\frac{1}{2}(\mu(B)+\mu(-B)),
\end{align*}
where $-B\coloneqq\{-x\colon x\in B\}$. We shall repeatedly use the elementary identities
\begin{align}
G_{\widetilde{\mu}_{|A|}}(z)&=z\tau((z^2-A^*A)^{-1}),\quad F_{\widetilde{\mu}_{|A|}}(z)=\frac{1}{z\tau((z^2-A^*A)^{-1})}, \label{eq:symm-modulus-cauchy}
\end{align}
for $A\in\widetilde{\MM}$ and $z\in\CC^+$. These follow immediately from the identity $\frac{1}{2}(\frac{1}{z-t}+\frac{1}{z+t})=\frac{z}{z^2-t^2}$.

The following fact provides the basic relation between $R$-diagonal operators and free additive convolution on the real line.

\begin{theorem}[\cite{NS1997rdiag}, \cite{HL2000Brown}, \cite{HS2007Brown}]
\label{thm:symmetrization}
Let $X_1,X_2\in\widetilde{\MM}$ be $*$-free, and assume that $X_2$ is $R$-diagonal. Then, for every $\lambda\in\CC$,
\begin{align*}
\widetilde{\mu}_{|X_1+X_2-\lambda|}=\widetilde{\mu}_{|X_1-\lambda|}\boxplus\widetilde{\mu}_{|X_2|}.
\end{align*}
\end{theorem}

We shall also use the following multiplicative property of $R$-diagonal elements.

\begin{theorem}[{\cite{NS1997rdiag,HL2000Brown,HS2007Brown}}]
\label{thm:rdiag-product}
Let $X_1,X_2\in\widetilde{\MM}$ be $*$-free, and assume that $X_2$ is $R$-diagonal. Then $X_1X_2$ is $R$-diagonal and
\begin{align*}
\mu_{|X_1X_2|^2}=\mu_{|X_1|^2}\boxtimes\mu_{|X_2|^2}.
\end{align*}
\end{theorem}

In view of Theorems~\ref{thm:subord_func} and~\ref{thm:symmetrization}, we let $\omega_1^{(\lambda)}$ and $\omega_2^{(\lambda)}$ denote the corresponding subordination functions, i.e., for every $z\in\CC^+$, we have
\begin{align}
F_{\widetilde{\mu}_{|X_1+X_2-\lambda|}}(z)&=F_{\widetilde{\mu}_{|X_1-\lambda|}}(\omega_1^{(\lambda)}(z))=F_{\widetilde{\mu}_{|X_2|}}(\omega_2^{(\lambda)}(z)), \label{eq:rdiag-subordination-F}\\
F_{\widetilde{\mu}_{|X_1+X_2-\lambda|}}(z)+z&=\omega_1^{(\lambda)}(z)+\omega_2^{(\lambda)}(z).
\label{eq:rdiag-subordination-sum}
\end{align}

The following formula expresses the regularized logarithmic potential of $X_1+X_2$ in terms of the subordination functions.

\begin{proposition}[{\cite[Proposition 3.4]{BZ2022}}]
\label{prop:log_symm_subord}
Let $X_1,X_2\in\widetilde{\MM}$ be $*$-free, and assume that $X_2$ is $R$-diagonal.
For every $\lambda\in\CC$ and $y>0$,
\begin{align*}
\frac{1}{2}\tau(\log(|X_1+X_2-\lambda|^2+y^2))
&=\frac{1}{2}\tau(\log(|X_1-\lambda|^2-\omega_1^{(\lambda)}(\iu y)^2))\\
&\quad +\frac{1}{2}\tau(\log(|X_2|^2-\omega_2^{(\lambda)}(\iu y)^2))\\
&\quad -\log(-\iu\omega_1^{(\lambda)}(\iu y)-\iu\omega_2^{(\lambda)}(\iu y)-y).
\end{align*}
\end{proposition}

\section{Freely infinitely divisible $R$-diagonal elements}
\label{sec:id-rdiagonal}

\subsection{Definitions and characterizations}

We now introduce infinite divisibility for $R$-diagonal elements. This notion was already formulated in \cite{BNNS2018eta} in the bounded, not necessarily tracial setting. Throughout the paper, free infinite divisibility means infinite divisibility with respect to additive free convolution. Since we work with the Fuglede--Kadison determinant and the Brown measure, we restrict attention to tracial $W^*$-probability spaces while allowing possibly unbounded operators.

\begin{definition}
An $R$-diagonal element $X\in\widetilde{\MM}$ is called freely infinitely divisible if, for every $n\ge1$, there exist a tracial $W^*$-probability space $(\NN,\tau_\NN)$ and freely independent identically distributed $R$-diagonal elements $X_1,\ldots,X_n\in\widetilde{\NN}$ such that
\begin{align*}
X\stard X_1+\cdots+X_n.
\end{align*}
\end{definition}

Let $(\MM,\tau)$ be a tracial $W^*$-probability space, and let $X\in\MM$ be a freely infinitely divisible $R$-diagonal element. By \cite[Remark 6.9]{BNNS2018eta}, there exists a unique probability measure $\sigma\in\PP_c^+$ such that
\begin{align*}
\mu_{XX^*}=\mu_{X^*X}=\BBP(\sigma)\boxtimes\Pi_1.
\end{align*}
Here $\PP_c^+$ denotes the set of compactly supported probability measures on $[0,\infty)$, $\BBP$ denotes the Boolean-to-free Bercovici--Pata bijection \cite{BP1999stable}, and $\Pi_1$ is the Marchenko--Pastur distribution with parameter $1$; see \cite[Section 6]{BNNS2018eta} for details. Conversely, for every $\sigma\in\PP_c^+$, there exists a freely infinitely divisible $R$-diagonal element $X$ satisfying the above relation.

The next characterization extends the preceding bounded result to the unbounded setting and is essentially known in a different context; see \cite{AHS2013law}. See also \cite{Bg2010surprising} for a related result.

\begin{theorem}
\label{thm:rid_char}
Let $X$ be an $R$-diagonal element, not necessarily bounded. Then the following are equivalent:
\begin{enumerate}
\item $X$ is freely infinitely divisible.
\item The symmetrized measure $\widetilde{\mu}_{|X|}$ is freely infinitely divisible.
\item There exists a probability measure $\sigma$ on $[0,\infty)$ such that
\begin{align*}
\mu_{XX^*}=\mu_{X^*X}=\BBP(\sigma)\boxtimes\Pi_1,
\end{align*}
where $\Pi_1$ is the Marchenko--Pastur distribution with parameter $1$.
\end{enumerate}
\end{theorem}

\begin{proof}
The equivalence (1)$\iff$(2) follows from the fact that the $*$-distribution of an $R$-diagonal element is determined by $\widetilde{\mu}_{|X|}$ and Theorem \ref{thm:symmetrization}. The equivalence (2)$\iff$(3) follows from \cite[Theorem 2.2 and Theorem 4.2]{AHS2013law}.
\end{proof}

\begin{remark}
The image $\BBP(\PP^+)$, where $\PP^+$ denotes the set of probability measures on $[0,\infty)$, is precisely the class of free regular measures; see \cite[Theorem 4.2]{AHS2013law}. Their properties are studied there in detail.
\end{remark}

We express the radial distribution function of the Brown measure in terms of the Voiculescu transform of $\widetilde{\mu}_{|X|}$.

\begin{proposition}
\label{prop:radial-voiculescu}
Let $X\in\log^+(\MM)$ be a non-zero freely infinitely divisible $R$-diagonal element. Let $(0,\Lambda)$ be the free generating pair of $\widetilde{\mu}_{|X|}$. For $q>0$, set
\begin{align*}
\theta(q)=1+\frac{1}{\iu q}\phi_{\widetilde{\mu}_{|X|}}(\iu q)=1-\int_{\RR}\frac{1+s^2}{q^2+s^2}\di{\Lambda}(s).
\end{align*}
Then there exists an open interval $I\subseteq\RR_{>0}$ such that $\theta \colon I\to({\widetilde{\mu}_{|X|}}(\{0\}),1)$ is an increasing bijection. Moreover, for $q\in I$,
\begin{align}
\mu_X(\{z\in\CC \colon |z|\le r(q)\})=\theta(q),\quad r(q)^2=q^2\theta(q)(1-\theta(q)).
\label{eq:radial-parametrization}
\end{align}
\end{proposition}

\begin{proof}
The existence of the interval $I$ can also be inferred from \cite[Example 3.4]{HW2022Reg}. We give a more direct argument here. Put $H(z)=z+\phi_{\widetilde{\mu}_{|X|}}(z)$ for $z\in\CC^+$. Then $H(\iu q)=\iu q+\phi_{\widetilde{\mu}_{|X|}}(\iu q)=\iu q\theta(q)$ for $q>0$. By dominated convergence, $\theta(q)\to1$ as $q\to\infty$. Moreover, differentiating gives
\begin{align*}
\theta^\prime(q)=2q\int_{\RR}\frac{1+s^2}{(q^2+s^2)^2}\di{\Lambda}(s)>0,\quad q>0.
\end{align*}
Thus $\theta$ is strictly increasing. Therefore
$I\coloneqq\{q>0 \colon \theta(q)>0\}$ is a nonempty open interval of the form $(q_0,\infty)$ for some $q_0\ge0$. We identify the lower endpoint of the range of $\theta$. If $q\in I$, then $\Im H(\iu q)=q\theta(q)>0$, so $\iu q\in\Omega_H$. Hence Theorem~\ref{thm:global_inversion} gives
\begin{align*}
F_{\widetilde{\mu}_{|X|}}(\iu q\theta(q))=F_{\widetilde{\mu}_{|X|}}(H(\iu q))=\iu q.
\end{align*}
Thus $\iu q\theta(q)G_{\widetilde{\mu}_{|X|}}(\iu q\theta(q))=\theta(q)$.  Observe that $q\theta(q)\to0$ as $q\to q_0$. Using the identity
$\widetilde{\mu}_{|X|}(\{0\})=\lim_{y\to0}\iu yG_{\widetilde{\mu}_{|X|}}(\iu y)$, we obtain
\begin{align*}
\lim_{q\to q_0}\theta(q)=\widetilde{\mu}_{|X|}(\{0\}).
\end{align*}
Since $\theta(q)\to1$ as $q\to\infty$ and $\theta$ is strictly increasing, it follows that $\theta\colon I\to(\widetilde{\mu}_{|X|}(\{0\}),1)$ is an increasing bijection.

It remains to identify the radial distribution. For $q\in I$, the identity $F_{\widetilde{\mu}_{|X|}}(\iu q\theta(q))=\iu q$ and the relation $G_{\widetilde{\mu}_{|X|}}(z)=zG_{\mu_{|X|^2}}(z^2)$ give
\begin{align*}
\iu q=F_{\widetilde{\mu}_{|X|}}(\iu q\theta(q))=\frac{\iu q\theta(q)}{1+\psi_{\mu_{|X|^2}}(-q^{-2}\theta(q)^{-2})}.
\end{align*}
Hence $\psi_{\mu_{|X|^2}}(-q^{-2}\theta(q)^{-2})=\theta(q)-1$, and so
$\chi_{\mu_{|X|^2}}(\theta(q)-1)=-q^{-2}\theta(q)^{-2}$. By the definition of the $S$-transform,
\begin{align*}
S_{\mu_{|X|^2}}(\theta(q)-1)=\frac{\theta(q)}{\theta(q)-1}\chi_{\mu_{|X|^2}}(\theta(q)-1)=\frac{1}{q^2\theta(q)(1-\theta(q))}.
\end{align*}
Thus, if $r(q)^2=q^2\theta(q)(1-\theta(q))$, then $S_{\mu_{|X|^2}}(\theta(q)-1)=r(q)^{-2}$. Applying Theorem~\ref{thm:char_r_supp_spec}, we get
\begin{align*}
\mu_X(\{z\in\CC \colon |z|\le r(q)\})=1+S_{\mu_{|X|^2}}^{\langle-1\rangle}(r(q)^{-2})=\theta(q).
\end{align*}
This proves the assertion.
\end{proof}

\subsection{Examples and basic constructions}

Theorem~\ref{thm:rid_char} provides a convenient way to produce examples and to prove stability under basic operations. We begin with a family containing the circular and circular Cauchy elements.

\begin{example}
\label{ex:x_mk}
Let $c_1,c_2,\ldots$ be a free circular family, and for $m\ge1$ and $k\ge0$ define
\begin{align*}
X_{m,k}\coloneqq\left(\prod_{i=1}^m c_i\right)\left(\prod_{j=1}^k c_{m+j}\right)^{-1}.
\end{align*}
By the multiplicative property of $R$-diagonal elements in Theorem~\ref{thm:rdiag-product}, the order of the factors does not affect the $*$-distribution. Thus any product obtained by permuting the $m$ circular elements and the $k$ inverse circular elements has the same $*$-distribution as $X_{m,k}$. The case $(m,k)=(1,0)$ gives a circular element, while $(m,k)=(1,1)$ gives the circular Cauchy element studied in \cite{HS2009inv}.

Using the identities (see, e.g., \cite[Section 5]{HM2013Law})
\begin{align*}
S_{\mu_{|c|^2}}(u)=S_{\Pi_1}(u)=\frac{1}{1+u},\quad S_{\mu_{|c^{-1}|^2}}(u)=-u,
\end{align*}
and the multiplicativity of the $S$-transform with respect to the free multiplicative convolution, together with Theorem \ref{thm:rdiag-product}, we obtain
\begin{align}
S_{\mu_{|X_{m,k}|^2}}(u)=\left(\frac{1}{1+u}\right)^m(-u)^k=\frac{(-u)^k}{(1+u)^m}.
\label{eq:S-Xmk}
\end{align}
On the other hand, recalling that $S_{f_{1/(k+1),1}}(u)=(-u)^k$ for the positive free stable law $f_{1/(k+1),1}$ of index $1/(k+1)$ (see, e.g., \cite[Proposition 3.5]{AH2016classical}), with the
convention $f_{1,1}=\delta_1$ when $k=0$, we deduce that
\begin{align*}
\mu_{|X_{m,k}|^2}=f_{1/(k+1),1}\boxtimes\Pi_1^{\boxtimes m}.
\end{align*}
By \cite[Remark 4.1(2), Theorem 4.2, Example 7.3(1)]{AHS2013law}, there exists $\sigma\in\PP^+$ such that $\BBP(\sigma)=f_{1/(k+1),1}$. Using the multiplicativity $\BBP(\sigma_1\boxtimes\sigma_2)=\BBP(\sigma_1)\boxtimes\BBP(\sigma_2)$ for $\sigma_1,\sigma_2\in\PP^+$ (see \cite{BN2008remarkable} for a proof) and $\BBP((\delta_0+\delta_2)/2)=\Pi_1$, it follows that $\mu_{|X_{m,k}|^2}=\BBP(\sigma^\prime)\boxtimes\Pi_1$ for some probability measure $\sigma^\prime \in\PP^+$. Therefore, by Theorem \ref{thm:rid_char}, we conclude that $X_{m,k}$ is $R$-diagonal and freely infinitely divisible for $m\ge1$ and $k\ge0$.

This family is also treated in \cite{KT2018limiting,COR2024fractional} with the notion of $\oplus$-stability. In particular, as stated in \cite[Remark~3.2]{KT2018limiting}, the Brown measure of $X_{m,k}$ is $\oplus$-stable if and only if $m=1$. In this case, the symmetrized law $\widetilde{\mu}_{|X_{1,k}|}$ is the symmetric free stable law of index $\frac{2}{k+1}$, and the corresponding $R$-transform satisfies $R_{\widetilde{\mu}_{|X_{1,k}|}}(\iu y)=\iu y^{(1-k)/(k+1)}$ for $y>0$.

\end{example}

\begin{remark}
If $a\in\CC\setminus\{0\}$ and $u$ is a Haar unitary, then $au$ is not freely infinitely divisible, since $\widetilde{\mu}_{|au|}=\frac{1}{2}\delta_{|a|}+\frac{1}{2}\delta_{-|a|}$ is not freely infinitely divisible.
\end{remark}

The next proposition shows that products, commutators, and anti-commutators of freely infinitely divisible $R$-diagonal elements are again freely infinitely divisible.

\begin{proposition}
\label{prop:R-diag-commutators}
Let $x$ and $y$ be $*$-free $R$-diagonal elements that are freely infinitely divisible, not necessarily bounded. Then $xy$ and $z_\pm\coloneqq xy\pm yx$ are again $R$-diagonal and freely infinitely divisible.
\end{proposition}

\begin{proof}
By Theorem \ref{thm:rid_char}, there exist probability measures $\sigma_x$ and $\sigma_y$ on $\RR_{\ge0}$ such that
\begin{align*}
\mu_{|x|^2}=\BBP(\sigma_x)\boxtimes\Pi_1,\quad \mu_{|y|^2}=\BBP(\sigma_y)\boxtimes\Pi_1.
\end{align*}
Therefore, we have
\begin{align*}
\mu_{|xy|^2}=\mu_{|x|^2}\boxtimes\mu_{|y|^2}=\BBP(\sigma_x)\boxtimes\Pi_1\boxtimes\BBP(\sigma_y)\boxtimes\Pi_1=\BBP\left(\sigma_x\boxtimes\sigma_y\boxtimes\frac{1}{2}(\delta_0+\delta_2)\right)\boxtimes\Pi_1,
\end{align*}
where the last equality follows from the multiplicativity of $\BBP$ and $\BBP((\delta_0+\delta_2)/2)=\Pi_1$.
The rest of the argument is similar to \cite[Section 5.3]{COR2024fractional}. Let $u$ be a Haar unitary that is $*$-free from $\{x,y\}$. Since $x$ and $y$ are $R$-diagonal, the pair $(ux,yu^*)$ has the same joint $*$-distribution as $(x,y)$, and therefore
\begin{align*}
z_\pm=xy\pm yx\stard(ux)(yu^*)\pm(yu^*)(ux)=uxyu^*\pm yx.
\end{align*}
Moreover, $uxyu^*$ and $yx$ are $*$-free and both are $R$-diagonal. Hence $z_\pm$ is $R$-diagonal. In particular,
\begin{align*}
\widetilde{\mu}_{|z_\pm|}=\widetilde{\mu}_{|uxyu^*|}\boxplus\widetilde{\mu}_{|yx|}=\widetilde{\mu}_{|xy|}^{\boxplus2}.
\end{align*}
Thus $\widetilde{\mu}_{|z_\pm|}$ is freely infinitely divisible, and Theorem \ref{thm:rid_char} implies that $z_\pm$ is $R$-diagonal and freely infinitely divisible.
\end{proof}

As a simple consequence, for freely independent circular elements $c_1$ and $c_2$, $c_1c_2\pm c_2c_1$ are $R$-diagonal and freely infinitely divisible.

In the next section, we show that this phenomenon persists for homogeneous
noncommutative polynomials in the bounded setting.

\section{Homogeneous polynomials of bounded freely infinitely divisible $R$-diagonal elements}
\label{sec:homogeneous-polynomials}

In the self-adjoint setting, several results for free infinite divisibility are known. For instance, the commutator of freely infinitely divisible random variables is freely infinitely divisible \cite{AHS2013law}, and certain quadratic forms in free random variables preserve free infinite divisibility \cite{LehnerEjsmont2017,LehnerEjsmont2021}. Furthermore, the infinite divisibility of powers of free random variables was studied in \cite{Hasebe2016}; in particular, every power of a semicircular element is freely infinitely divisible. See also \cite{lehnerkamil2024integral,popakamil2025} for more recent work on free infinite divisibility of polynomials of free random variables.

The main result of this section is the following theorem.

\begin{theorem}\label{thm:prod}
Let $(\mathcal M,\tau)$ be a tracial $W^*$-probability space, and let $X_1,\dots,X_N\in\mathcal M$ be a $*$-free family of bounded freely infinitely divisible $R$-diagonal elements. Fix $m\geq1$. For $\iota=(\iota(1),\dots,\iota(m))\in\{1,\dots,N\}^m$, set $X^\iota=X_{\iota(1)}X_{\iota(2)}\cdots X_{\iota(m)}$.
Then, for every function $\alpha \colon \{1,\dots,N\}^m\to\CC$, the element
\begin{align*}
T=\sum_{\iota\in\{1,\dots,N\}^m}\alpha(\iota)X^\iota
\end{align*}
is also $R$-diagonal and freely infinitely divisible.
\end{theorem}

We first introduce a Fock space model and some auxiliary notation used in the proof of Theorem~\ref{thm:prod}. For the proof of free infinite divisibility, by a standard approximation argument using symmetric compound free Poisson laws, it is enough to treat first the case where each symmetrization \(\widetilde{\mu}_{|X_k|}\) is a symmetric compound free Poisson distribution. In other words, we can assume that $X_1,\dots,X_N$ are a $*$-free family of $R$-diagonal elements whose determining sequences are represented by compactly supported symmetric finite measures $\sigma_1,\dots,\sigma_N$ on $\RR$, in the sense that
\begin{align}
\kappa_r[h_k]=\kappa_r[X_k^{\varepsilon(1)},\dots,X_k^{\varepsilon(r)}]=\int_{\RR}t^r\di{\sigma_k}(t)
\label{eq:compound-model}
\end{align}
where the sequence $\varepsilon(1),\dots,\varepsilon(r)$ is cyclically alternating and $h_k$ is a self-adjoint operator whose distribution is $\widetilde{\mu}_{|X_k|}$. Let $\{e,f\}$ be an orthonormal basis, and for $1\leq k\leq N$, set
\begin{align*}
H_k^e=L^2(\sigma_k)\otimes\CC e,\quad H_k^f=L^2(\sigma_k)\otimes\CC f,
\end{align*}
and
\begin{align*}
\mathcal H=\bigoplus_{k=1}^N(H_k^e\oplus H_k^f),\quad \widetilde{\mathcal H}=\bigoplus_{n=1}^{\infty}\mathcal H^{\otimes n}.
\end{align*}
Note that $\widetilde{\mathcal H}$ is the full Fock space without the vacuum state. Let $\mathrm{id}_k$ denote the function $t\mapsto t$ in $L^2(\sigma_k)$. With this notation, define the operator $m_k\colon H_k^e\to H_k^f$ by
\begin{align*}
m_k(F\otimes e)=(\mathrm{id}_k \cdot F)\otimes f,
\end{align*}
and the operators $M_k,A_k^*,D_k^*$ on $\widetilde{\mathcal H}$ by linear extension of
\begin{align*}
M_k(v_1\otimes\cdots\otimes v_n)&=m_k(P_{H_k^e}v_1)\otimes v_2\otimes\cdots\otimes v_n,\\
A_k^*(v_1\otimes\cdots\otimes v_n)&=(\mathrm{id}_k\otimes f)\otimes v_1\otimes\cdots\otimes v_n,\\
D_k^*(v_1\otimes\cdots\otimes v_n)&=(\mathrm{id}_k\otimes e)\otimes v_1\otimes\cdots\otimes v_n,
\end{align*}
where $P_{H_k^e}$ denotes the orthogonal projection from $\mathcal H$ onto $H_k^e$ and $v_1,\ldots,v_n\in\mathcal{H}$. The adjoint of $M_k$ is given by
\begin{align*}
M_k^*(v_1\otimes\cdots\otimes v_n)&=m_k^*(P_{H_k^f}v_1)\otimes v_2\otimes\cdots\otimes v_n.
\end{align*}
The adjoints of $A_k^*$ and $D_k^*$ on $\widetilde{\mathcal H}$ are given by
\begin{align*}
A_k(v_1\otimes\cdots\otimes v_n)&=\langle v_1,\mathrm{id}_k\otimes f\rangle v_2\otimes\cdots\otimes v_n,\\
D_k(v_1\otimes\cdots\otimes v_n)&=\langle v_1,\mathrm{id}_k\otimes e\rangle v_2\otimes\cdots\otimes v_n
\end{align*}
for $n\geq2$, while for $v\in\mathcal H$ we set $A_kv=0,\ D_kv=0$. Here $P_{H_k^f}$ denotes the orthogonal projection from $\mathcal H$ onto $H_k^f$. For a positive integer $R$, let $NC_{\mathrm{irr}}(R)$ be the set of noncrossing partitions of $\{1,\dots,R\}$ with exactly one exterior block, that is, such that $1$ and $R$ are in the same block of $\pi$. We write $\mathbbm{1}_R$ for the one-block partition of $\{1,\dots,R\}$.

\begin{lemma}
\label{lemma:prod:1}
Let $R$ be a positive integer. Let $\varepsilon\colon \{1,\dots,R\}\to\{1,*\}$ and let $\iota\colon \{1,\dots,R\}\to\{1,\dots,N\}$ be such that $\iota(1)=\iota(R)=k$. Let $\pi\in NC(R)$ be such that
\begin{enumerate}
\item[$(i)$] $\pi\in NC_{\mathrm{irr}}(R)$.
\item[$(ii)$] If $ B = (q_1, q_2, \dots, q_r) $ is a block of $ \pi $, then $\iota (q_1) = \iota(q_2) = \dots =\iota(q_r) $ and $ \varepsilon(q_1) \neq \varepsilon(q_2) \neq \dots \neq \varepsilon(q_r)  \neq \varepsilon(q_1)$.
\end{enumerate}
Define $\varphi_\pi\colon \{1,\dots,R\}\to\mathcal L(\widetilde{\mathcal H})$ as follows:
\begin{itemize}
\item 
If $j\in\{1,R\}$, set $\varphi_\pi(j)=M_{\iota(j)}^{\varepsilon(j)}$.

\item 
If $2\leq j\leq R-1$ and $\varepsilon(j)=1$, set
\begin{align*}
\varphi_\pi(j)=
\begin{cases}
D_{\iota(j)},&\text{if $j$ is the least element of a block of $\pi$,}\\
A_{\iota(j)}^*,&\text{if $j$ is the greatest element of a block of $\pi$,}\\
M_{\iota(j)},&\text{otherwise.}
\end{cases}
\end{align*}

\item
If $2\leq j\leq R-1$ and $\varepsilon(j)=*$, set
\begin{align*}
\varphi_\pi(j)=
\begin{cases}
A_{\iota(j)},&\text{if $j$ is the least element of a block of $\pi$,}\\
D_{\iota(j)}^*,&\text{if $j$ is the greatest element of a block of $\pi$,}\\
M_{\iota(j)}^*,&\text{otherwise.}
\end{cases}
\end{align*}
\end{itemize}

Let $1_k$ be the constant function $1$ in $L^2(\sigma_k)$, and set $\eta_k=1_k\otimes e+1_k\otimes f$. Then
\begin{align}
\kappa_\pi[X_{\iota(1)}^{\varepsilon(1)},\dots,X_{\iota(R)}^{\varepsilon(R)}]=\langle\varphi_\pi(1)\cdots\varphi_\pi(R)\eta_k,\eta_k\rangle.
\label{eq:lemma_phi}
\end{align}
\end{lemma}

 \begin{proof}
We prove \eqref{eq:lemma_phi} by induction on the number of blocks of $\pi$. Suppose first that $\pi=\mathbbm{1}_R$. Then condition $(ii)$ implies that $\varepsilon(1),\dots,\varepsilon(R)$ is cyclically alternating and $\iota(1)=\cdots=\iota(R)=k$. Hence, by \eqref{eq:compound-model},
\begin{align*}
\kappa_\pi[X_{\iota(1)}^{\varepsilon(1)},\dots,X_{\iota(R)}^{\varepsilon(R)}]=\int_{\RR}t^R\di{\sigma_k}(t).
\end{align*}
On the other hand,
\begin{align*}
\langle\varphi_\pi(1)\cdots\varphi_\pi(R)\eta_k,\eta_k\rangle=\langle M_k^{\varepsilon(1)}\cdots M_k^{\varepsilon(R)}\eta_k,\eta_k\rangle=\int_{\RR}t^R\di{\sigma_k}(t).
\end{align*}
Thus the assertion holds when $\pi$ has one block.

Assume now that $\pi$ has more than one block. Since $\pi\in NC_{\mathrm{irr}}(R)$, there exists an interval block
\begin{align*}
B=\{q+1,q+2,\dots,q+l\}
\end{align*}
which does not contain $1$ or $R$. Write $\iota(q+1)=\cdots=\iota(q+l)=s$. By condition $(ii)$, the signs on $B$ are cyclically alternating. If $\varepsilon(q+l)=1$, then $\varepsilon(q+1)=*$, and
\begin{align*}
\varphi_\pi(q+1)\cdots\varphi_\pi(q+l)\xi=\left(\int_{\RR}t^l\di{\sigma_s}(t)\right)\xi
\end{align*}
for every $\xi\in\widetilde{\mathcal H}$. Indeed, the rightmost operator creates $\mathrm{id}_s\otimes f$, the middle operators multiply by the appropriate powers of $t$, and the leftmost operator annihilates the created letter. The case $\varepsilon(q+l)=*$ is similar, with $e$ and $f$ interchanged. Therefore
\begin{align*}
\varphi_\pi(1)\cdots\varphi_\pi(R)\eta_k=\left(\int_{\RR}t^l\di{\sigma_s}(t)\right)\varphi_\pi(1)\cdots\varphi_\pi(q)\varphi_\pi(q+l+1)\cdots\varphi_\pi(R)\eta_k.
\end{align*}
On the cumulant side, by multiplicativity over blocks,
\begin{align*}
\kappa_\pi[X_{\iota(1)}^{\varepsilon(1)},\dots,X_{\iota(R)}^{\varepsilon(R)}]=\left(\int_{\RR}t^l\di{\sigma_s}(t)\right)\kappa_{\pi^\prime}[X_{\iota^\prime(1)}^{\varepsilon^\prime(1)},\dots,X_{\iota^\prime(R-l)}^{\varepsilon^\prime(R-l)}],
\end{align*}
where $\pi^\prime$ and the functions $\iota^\prime,\varepsilon^\prime$ are obtained by deleting the interval block $B$ and relabeling the remaining ordered set. The partition $\pi^\prime$ still satisfies conditions $(i)$ and $(ii)$. Applying the induction hypothesis to $\pi^\prime$ gives the desired identity.
\end{proof}

\begin{lemma}\label{lemma:prod:2}
 
 Let $ R $ be a positive integer and let
 $ \varepsilon\colon\{1, 2, \dots, R \} \to \{1, \ast\} $ and 
 $\iota\colon\{1, 2, \dots, R\} \to \{1, 2, \dots, N\} $ be two mappings.
 
  Suppose that 
  $ \psi\colon\{1, 2, \dots, R \} \to \mathcal{L}(\widetilde{\mathcal{H}}) $ is such that:
  \begin{enumerate}
  \item[(a)] $ \psi(1) \in \{ M_{\iota(1)}, A_{\iota(1)}^\ast \} $ if $ \varepsilon(1) = 1 $, 
  respectively
  $ \psi(1) \in \{ M_{\iota(1)}^\ast , D_{\iota(1)}^\ast \} $ if $ \varepsilon(1) = \ast $
  \item[(b)] $ \psi(R) \in \{ M_{\iota(R)}, D_{\iota(R)} \} $ if $ \varepsilon(R) = 1 $,
  respectively
  $ \psi(R) \in \{ M_{\iota(R)}^\ast,  A_{\iota(R)}   \} $  if $ \varepsilon(R) = \ast $
  \item[(c)] if $ 2 \leq j \leq R -1 $ then $ \psi(j) \in \{  M_{\iota(j)}^{\varepsilon(j)}, (A^\ast_{\iota(j)})^{\varepsilon(j)}, D_{\iota(j)} ^{\varepsilon(j)} \}$ 
  \end{enumerate}
  
  Let $ k, k ^\prime \in \{1, 2, \dots, N\}$.
  Then 
  $ \langle \psi(1) \psi(2) \cdots \psi(R) \eta_k, \eta_{k^\prime} \rangle =0 $
  unless $k =k^\prime = \iota(1) =\iota(R) $ and there exists a partition $ \pi \in NC(R) $ satisfying conditions $(i)$ and $(ii)$ from Lemma \ref{lemma:prod:1} such that 
  $ \psi = \varphi_{\pi} $.
\end{lemma}

\begin{proof}
We shall show the result by induction on $R$. Suppose that
\begin{align}
\langle\psi(1)\cdots\psi(R)\eta_k,\eta_{k'}\rangle\neq0.
\label{eq:nonzero-psi}
\end{align}
Since the annihilation operators $A_i,D_i$ vanish on $\mathcal H$, $\psi(R)$ cannot be an annihilation operator. Hence $\psi(R)=M_{\iota(R)}^{\varepsilon(R)}$, and $\psi(R)\eta_k\neq0$ implies $\iota(R)=k$. Similarly, $\psi(1)$ cannot be a creation operator, because otherwise $\psi(1)\cdots\psi(R)\eta_k$ has tensor length at least $2$ and is orthogonal to $\eta_{k'}$. Hence $\psi(1)=M_{\iota(1)}^{\varepsilon(1)}$.

For $R=2$, \eqref{eq:nonzero-psi} gives
\begin{align*}
\langle M_{\iota(1)}^{\varepsilon(1)}M_{\iota(2)}^{\varepsilon(2)}\eta_k,\eta_{k'}\rangle\neq0.
\end{align*}
This is possible only if $\iota(1)=\iota(2)=k=k'$ and $\varepsilon(1)\neq\varepsilon(2)$. Thus the assertion holds with $\pi=\mathbbm{1}_2$.

Assume now that $R>2$. Let
\begin{align*}
\mathcal A=\{j\colon \psi(j)\in\{A_{\iota(j)},D_{\iota(j)}\}\},\quad \mathcal B=\{j\colon \psi(j)\in\{A_{\iota(j)}^*,D_{\iota(j)}^*\}\}.
\end{align*}
We distinguish three cases:

\textbf{Case 1}:   $ \mathcal{A} = \emptyset = \mathcal{B} $. 

Then $\psi(j)=M_{\iota(j)}^{\varepsilon(j)}$ for every $j$. Condition \eqref{eq:nonzero-psi} gives, as above,
$ \iota(1)=\cdots=\iota(R)=k=k'$ and $\varepsilon(1)\neq\varepsilon(2)\neq\cdots\neq\varepsilon(R)\neq\varepsilon(1)$. Hence $\psi=\varphi_\pi$ for $\pi=\mathbbm{1}_R$.

\textbf{Case 2}: $\mathcal{B} = \emptyset \neq \mathcal{A} $.

Suppose next that $\mathcal B=\emptyset$ and $\mathcal A\neq\emptyset$. Let $s=\max\mathcal A$. Then $\psi(j)$ is a multiplication operator for every $j>s$, so $\psi(s+1)\cdots\psi(R)\eta_k\in\mathcal H$. Since $\psi(s)$ is an annihilation operator, this contradicts \eqref{eq:nonzero-psi}. Hence this case cannot occur.

\textbf{Case 3}: $ \mathcal{B} \neq \emptyset $.

Let $s=\min\mathcal B$. If $\mathcal A\cap\{1,\dots,s-1\}=\emptyset$, then all operators to the left of $\psi(s)$ are multiplication operators. Since $\psi(s)$ is a creation operator, the final vector has tensor length at least $2$, contradicting \eqref{eq:nonzero-psi}. Thus $\mathcal A\cap\{1,\dots,s-1\}\neq\emptyset$. Let $ t = \max(\mathcal{A} \cap \{1, 2, \dots, s-1\} ) $.  To simplify the notation, let $ q \geq 0 $ and $ l \geq 2 $ be such that $ t = q+1 $ and $ s = q + l $. 
    In particular, $ \psi(j) = M_{\iota(j)}^{\varepsilon(j)} $ for all $ j  $ such that $ q+1 < j < q+l $.
    
     If $ \varepsilon(q+l) = 1 $, then 
     $ q+l \in \mathcal{B} $ gives that $ \psi(q+l ) =  A^\ast_{\iota(q+l )} $ 
     and the condition
      $ \psi(q+1)\psi(q+2) \cdots \psi(q+l) \neq 0 $ 
      gives that
       $\iota(q+1) = \dots = \iota(q+l) $, 
      $ \varepsilon(q+1) \neq\varepsilon(q+2) \neq \dots \neq \varepsilon(q+l) \neq \varepsilon(q+1) $ 
      and $\psi(q+1) = A_{\iota(q+1) } $.
      Therefore
\[
\psi(q+1)\cdots\psi(q+l)=\left(\int_{\RR}t^l\di{\sigma_{\iota(q+1)}}(t)\right)\cdot \mathrm{Id},
\]
      which gives 
      \[   \psi(1) \cdots \psi(R) = \kappa_l(X_{\iota(q+1)} , X_{\iota(q+1)}^\ast, \dots )\cdot  \psi(1) \cdots \psi(q)  \psi(q+l+1) \cdots \psi(R) .
      \]
Hence, after deleting the interval \(\{q+1,\dots,q+l\}\), the resulting
sequence of operators still satisfies the same assumptions, and the corresponding
inner product is still nonzero. Thus the conclusion follows from the induction
hypothesis. The case \(\varepsilon(q+l)=\ast\) is similar.
\end{proof}

We introduce some notation used in the proof of Theorem~\ref{thm:prod}. Set $\eta=\sum_{k=1}^N \eta_k$. Let \(R\ge2\), and let $\varepsilon\colon\{1,\dots,R\}\to\{1,\ast\},\quad
\iota\colon\{1,\dots,R\}\to\{1,\dots,N\}$.
Define a map
$\psi_{\iota,\varepsilon}\colon \{1,\dots,R\}\to\mathcal L(\widetilde{\mathcal H})$
as follows:
\begin{align*}
 \psi_{\iota,\varepsilon}(1) &=
 \begin{cases}
 M_{\iota(1)}+A_{\iota(1)}^\ast, & \varepsilon(1)=1,\\
 M_{\iota(1)}^\ast+D_{\iota(1)}^\ast, & \varepsilon(1)=\ast,
 \end{cases}\\
 \psi_{\iota,\varepsilon}(R) &=
 \begin{cases}
 M_{\iota(R)}+D_{\iota(R)}, & \varepsilon(R)=1,\\
 M_{\iota(R)}^\ast+A_{\iota(R)}, & \varepsilon(R)=\ast,
 \end{cases}\\
 \psi_{\iota,\varepsilon}(j)
 &=M_{\iota(j)}^{\varepsilon(j)}
 +(A_{\iota(j)}^\ast)^{\varepsilon(j)}
 +D_{\iota(j)}^{\varepsilon(j)} \text{ for } 1<j<R.
\end{align*}
We also set
\[
\Psi(\iota,\varepsilon)
=
\psi_{\iota,\varepsilon}(1)\psi_{\iota,\varepsilon}(2)\cdots
\psi_{\iota,\varepsilon}(R).
\]

After expanding \(\Psi(\iota,\varepsilon)\), Lemmas \ref{lemma:prod:1}
and \ref{lemma:prod:2} imply the following identity.
\begin{lemma}
\[ \sum_{\pi\in NC_{\mathrm{irr}}(R)} \kappa_\pi \big[  
X_{\iota(1)}^{\varepsilon(1)}, X_{\iota(2)}^{\varepsilon(2)}, \dots, X_{\iota(R)}^{\varepsilon(R)} \big] = \langle \Psi(\iota, \varepsilon) \eta, \eta \rangle.
\]
\end{lemma}

\begin{proof}[Proof of Theorem \ref{thm:prod}]
First, we show that $T$ is $R$-diagonal. It is enough to prove that, for every $n\geq1$ and every map
$\ee\colon\{1,\dots,n\}\to\{1,\ast\}$,
 \begin{equation}\label{eq:01}
  \kappa_n \left( T^{\ee(1)}, T^{\ee(2)}, \dots, T^{\ee(n)} \right) = 0 
  \end{equation}
 unless $ \ee(1)\neq\ee(2)\neq\dots\neq\ee(n)\neq\ee(1)$.
 
  The left-hand side of \eqref{eq:01} 
   is a linear combination of terms of the type 
 \begin{equation}\label{k:01}
 \kappa_n\left( 
 Y_1^{\theta(1)} Y_2^{\theta(2)} \cdots Y_m^{\theta(m)}, 
  Y_{m+1}^{\theta(m+1)}\cdots Y_{2m}^{\theta(2m)}, 
  \dots, 
  Y_{m(n-1) + 1}^{\theta(m(n-1)+1)} \cdots Y_{nm}^{\theta(nm)}
 \right),
 \end{equation}
where
$ \theta(m(s-1) + t) = \ee(s)$ for $ 1\leq s \leq n $ and $ 1 \leq t \leq m $, and
\[
Y_{m(s-1) + t } =
\left\{
\begin{array}{ll}
X_{\iota_s(t)} & \text{if } \ee(s) = 1,\\
X_{\iota_s(m+1-t)}  &\text{if } \ee(s) = \ast,
\end{array}
\right.
\]
for some $\iota_s\in \{1,2,\dots,N\}^m$.
 The formula for free cumulants with products as entries (see, e.g., \cite[Theorem 11.12]{NicaSpeicher2006book}) gives that the free cumulant (\ref{k:01}) equals
 \[
 \sum_{\substack{\pi\in NC(mn)\\ \pi\vee  \widehat{0}_n = \mathbbm{1}_{mn}}}
 \kappa_{\pi} \left[ 
 Y_1^{\theta(1)}, Y_2^{\theta(2)}, \dots, Y_{mn}^{\theta(mn)}
 \right],
 \]
where $\widehat{0}_n=\{\{1,2,\dots,m\}, \dots, \{m(n-1)+1,\dots,mn\}\}$. 
Let $ \pi\in NC(mn) $ be such that $ \pi\vee  \widehat{0}_n = \mathbbm{1}_{mn}$ and 
 \[
   \kappa_{\pi} \big[ 
 Y_1^{\theta(1)}, Y_2^{\theta(2)}, \dots, Y_{mn}^{\theta(mn)}
 \big] \neq 0 .
 \]
 Since $X_1,\dots,X_N$ are free $R$-diagonal elements, every block $(i_1,\dots,i_q)$ of $\pi$ satisfies
 $ Y_{i_1} = Y_{i_2} = \dots = Y_{i_q} $ and
 $ \theta(i_1)\neq\theta(i_2)\neq\dots\neq\theta(i_q)\neq\theta(i_1)$.
 
Let  
$ t \in \{ 1, 2, \dots, n \} $. 
 Consider the block $ B $ of $ \pi $ that contains $ tm $, and let $ p $ be the next element of $ B $ after $ tm $ in the cyclic order. Note that the set $\{ tm+1, tm+2, \dots, p-1\} $ is a union of blocks of $\pi $. Hence the set 
$\{ \theta(tm+1), \theta(tm+2), \dots, \theta(p-1) \} $
 has the same number of elements that equal $ 1 $ as the number of elements that equal $\ast$.
Since $\theta$ is constant on each interval
$\{m(s-1)+1,\dots,ms\}$, this can happen only if $p$ is the first point of one of these intervals; otherwise the last partial interval would contribute a nonzero number, strictly smaller than $m$, of equal signs. Thus $p=ms+1$ for some $s$, where $mn+1$ is understood as $1$

We claim that in fact $p=tm+1$. If not, then the cyclic interval between $tm+1$ and $p-1$ is a nonempty proper union of blocks of $\pi$ and also a union of blocks of $\widehat{0}_n$. This contradicts
$ \pi\vee\widehat{0}_n=\mathbbm{1}_{mn} $. Hence $ tm $ and $ tm+1 $ are consecutive elements in $ B $. Finally, using again the $ R $-diagonality of $ X_1, \dots, X_N $, we obtain that $ \theta(tm) \neq \theta(tm+1) $, which gives that $ \ee(t) \neq \ee(t+1) $, where $\ee(n+1)=\ee(1)$. This proves \eqref{eq:01}.
 
 In particular, we showed that 
  $\displaystyle  \kappa_{\pi} \big[ 
 Y_1^{\theta(1)}, Y_2^{\theta(2)}, \dots, Y_{mn}^{\theta(mn)}
 \big] \neq 0  $ and $ \pi\vee  \widehat{0}_n = \mathbbm{1}_{mn}$
  if and only if 
 \[
 \displaystyle  \kappa_{\pi} \big[ 
 Y_1^{\theta(1)}, Y_2^{\theta(2)}, \dots, Y_{mn}^{\theta(mn)}
 \big] \neq 0  
 \]
 and 
 $tm, tm+1 $ are in the same block of $ \pi $ for each $ t \in \{1, 2, \dots, n\} $, where $mn+1$ is understood as $1$. The converse implication is immediate, since these relations connect the $n$ product blocks cyclically.
 
 Next, we shall show that $ T $ is freely infinitely divisible. 
 
 Let $ V \in \mathcal{L}(\widetilde{\mathcal{H}}) $ be given by
 \[
 \displaystyle V =\sum_{\iota\in \{1, 2, \dots, N\}^m} \alpha(\iota)\Psi(\iota, 1).
 \]
 Here, in the second argument of $\Psi$, the symbols $1$ and $\ast$ denote the constant maps with values $1$ and $\ast$, respectively. Notice that, by the definition of $\Psi$,
 \[
 \Psi(\iota,1)^*=\Psi(\widehat{\iota},\ast),
 \qquad
 \widehat{\iota}(j)=\iota(m+1-j).
 \]
 
 We shall show that 
 \[
 \kappa_{2n}\big(T, T^\ast, \dots, T, T^\ast\big) = \big\langle( VV^\ast )^n\eta, \eta\big\rangle.
 \]
 
 For this, it suffices to show that, for any $\iota_1, \dots, \iota_{2n}\colon\{1, 2, \dots, m\} \to \{1, 2, \dots, N\} $,
 \begin{equation}\label{main:1}
 \begin{aligned}
 &\kappa_{2n}\big( X^{\iota_1}, (X^{\iota_2})^\ast, \dots, X^{\iota_{2n-1}}, (X^{\iota_{2n}})^\ast  
 \big) \\
 &\quad= \big\langle \Psi(\iota_1, 1)\Psi(\widehat{\iota_2}, \ast)\cdots
 \Psi(\iota_{2n-1},1)\Psi(\widehat{\iota_{2n}}, \ast)
 \eta, \eta\big\rangle .
 \end{aligned}
 \end{equation}
 
 The left-hand side of (\ref{main:1}) equals 
  \[
   \sum \kappa_{\pi} \big[ 
 Y_1^{\theta(1)}, Y_2^{\theta(2)}, \dots, Y_{2mn}^{\theta(2mn)}
 \big],
 \]
 where $ Y_j $ and $ \theta(j) $ are defined as before, and the sum is taken over all
 $ \pi \in NC(2mn) $ such that $\pi $ connects $ tm $ and $ tm+1 $ for all $ t \in \{1, 2, \dots, 2n\} $, where $2mn+1$ is understood as $1$.
 
 The right-hand side of (\ref{main:1}) expands as a sum of products of operators of the types $M_k, A_k, D_k $ and their adjoints ($1 \leq k \leq N$). Lemmas \ref{lemma:prod:1} and \ref{lemma:prod:2} show that the terms with non-vanishing contribution with respect to the state $\langle \cdot \eta, \eta \rangle $ are of the form 
 $  \varphi_{\pi}(1) \varphi_{\pi}(2 )    \cdots \varphi_{\pi}(2mn)  $
  with $ \pi \in NC(2mn) $ satisfying conditions $(i)$ and $(ii) $ from Lemma \ref{lemma:prod:1}. It therefore suffices to show that:
 \begin{enumerate}
 \item[(a)] If $ \pi \in NC(2mn) $ connects $ tm $ and $ tm+1 $ for each
  $ t \in \{1, 2, \dots, 2n\} $,
   then $ \varphi_{\pi}(1) \varphi_{\pi}(2 )    \cdots \varphi_{\pi}(2mn) $ is a term in the expansion of the right-hand side of (\ref{main:1}).
 \item[(b)] If $ \pi \in NC(2mn) $ is such that
  $  \varphi_{\pi}(1) \varphi_{\pi}(2 )    \cdots \varphi_{\pi}(2mn)  $
   is a term in the expansion of the right-hand side of (\ref{main:1}), then
    $\big\langle  \varphi_{\pi}(1)    \cdots \varphi_{\pi}(2mn)   \eta, \eta \big\rangle= 0 $ 
    unless $ \pi $ connects $ tm $ and $ tm+1 $ for all $ t \in\{1, 2, \dots, 2n\} $.
 \end{enumerate}
 
 For (a), note first that if $ t<2n $ and $ tm $ and $ tm + 1 $ are in the same block of $ \pi $, then $ tm $ is not the greatest element in its block, nor is $ tm+1 $ the least element in its block. The construction of $ \varphi_\pi $ then gives $\varphi_\pi (tm) \notin  \{ A_k^\ast, D_k^\ast :\ 1 \leq k \leq N\}$ and $\varphi_\pi(tm+1) \notin  \{  A_k, D_k : \ 1 \leq k \leq N\}$.
These are exactly the endpoint restrictions in the two adjacent factors of the right-hand side of (\ref{main:1}). For $t=2n$, the condition says that $2mn$ and $1$ are in the same block; by the special definition of $\varphi_\pi$ at the two endpoints, both corresponding operators are multiplication operators, which are also allowed. Hence the conclusion follows.
 
 For (b), let $ \pi \in NC(2mn) $ be such that  $  \varphi_{\pi}(1) \varphi_{\pi}(2 )    \cdots \varphi_{\pi}(2mn)  $
   is a term in the expansion of the right-hand side of (\ref{main:1}) and 
   $\big\langle  \varphi_{\pi}(1)    \cdots \varphi_{\pi}(2mn)   \eta, \eta \big\rangle\neq 0 $. 
   From Lemma \ref{lemma:prod:2}, $ \pi $ connects $ 1 $ and $ 2mn $. 
   Fix $ t < 2n $ and suppose that $ tm, tm+1 $ are in different blocks of $\pi $. Since $ \pi $
   is non-crossing, either $ tm $ is the greatest element of some block of $ \pi $, or $ tm+1 $ is the least element of some block of $ \pi $. Indeed, if neither of these happened, then the block of $tm$ would have a point to the right of $tm+1$, while the block of $tm+1$ would have a point to the left of $tm$, producing a crossing. Hence either $\varphi_\pi(tm) \in \{  A_k^\ast, D_k^\ast: 1 \leq k \leq N  \}$ or $\varphi_\pi(tm+1) \in \{ A_k, D_k : 1 \leq k \leq N \}$, both contradicting the endpoint restrictions in the expansion of the right-hand side of (\ref{main:1}). Therefore $\pi$ connects $tm$ and $tm+1$ for all $t\in\{1,\dots,2n\}$.
This proves (\ref{main:1}), and hence
\[
 \kappa_{2n}\big(T, T^\ast, \dots, T, T^\ast\big)
 =
 \big\langle( VV^\ast )^n\eta, \eta\big\rangle.
\]

Since \(V\) is bounded, the preceding identity shows that the determining sequence of \(T\) is represented by a compactly supported symmetric finite measure. Thus \(\widetilde{\mu}_{|T|}\) is freely infinitely divisible in the compound free Poisson case. The general bounded case follows by a standard approximation argument, using the weak closedness of freely infinitely divisible laws and the closure of $R$-diagonality under convergence in $*$-distribution. Hence \(T\) is freely infinitely divisible, completing the proof.
\end{proof}

\section{Analysis of the support for freely infinitely divisible $R$-diagonal elements}
\label{sec:support_infinitely_divisible_rdiag}

We first recall a result on general $R$-diagonal perturbations that will be used below. This result belongs to the general $R$-diagonal framework and will be proved in the forthcoming revised version of \cite{BZ2022}.

\begin{theorem}[{\cite{BZ2022}, forthcoming revised version}]
\label{thm:positivity}
Assume that $X_0\in \log^+(\MM)\setminus\CC$ and $Y\in \log^+(\MM)\setminus\CC$ are $*$-free and that $Y$ is $R$-diagonal. Let $X=X_0+Y$ and let
\begin{align*}
\Omega=\CC\setminus(S\cup F_1\cup F_2),
\end{align*}
where
\begin{align*}
S&=\{\lambda\in\CC\colon \tau(\ker(X_0-\lambda))+\tau(\ker Y)\ge 1\},\\
F_1&=\left\{\lambda\in\CC\colon \mom_2(Y)\le \frac{1}{\mom_{-2}(X_0-\lambda)}\right\},\\
F_2&=\left\{\lambda\in\CC\colon \mom_2(X_0-\lambda)\le \frac{1}{\mom_{-2}(Y)}\right\}.
\end{align*}
Then $\Omega$ is open, the Brown measure $\mu_X$ is absolutely continuous on $\Omega$, and its density is strictly positive on $\Omega$. Moreover, $\supp(\mu_X)=\cl{\Omega}$.
\end{theorem}

In the infinitely divisible case, the characterization from Theorem \ref{thm:rid_char} gives additional control of the quantities that define the region $\Omega$. This allows us to derive results on support and regularity from Theorem \ref{thm:positivity}.

\subsection{Coincidence of the support and the spectrum}

We begin with the negative second moment $\mom_{-2}(X)$ and then prove the coincidence of the support and the spectrum for bounded freely infinitely divisible $R$-diagonal elements.

\begin{proposition}
\label{prop:l-2_infty}
Let $T$ be a freely infinitely divisible $R$-diagonal element. Then
\begin{align*}
\mom_{-2}(T)=\int_0^\infty\frac{1}{x}\di{\mu_{|T|^2}}(x)=\infty.
\end{align*}
\end{proposition}

\begin{proof}
If $\tau(\ker T)>0$, then the assertion follows from the definition of $\mom_{-2}(T)$. Thus we may assume that $\mu_{|T|^2}(\{0\})=0$. By Theorem \ref{thm:rid_char}, there exists a probability measure
$\nu\in\PP^+$ such that $\mu_{|T|^2}=\nu\boxtimes\Pi_1$. Then $\nu(\{0\})=0$. By Lemma \ref{lem:s_transform_rel}, it is enough to show that $\lim_{u\to-1}S_{\mu_{|T|^2}}(u)=\infty$. Indeed, by the multiplicativity of the $S$-transform, we have
\begin{align*}
\lim_{u\to-1}S_{\mu_{|T|^2}}(u)=\lim_{u\to-1}S_\nu(u)S_{\Pi_1}(u)=\lim_{u\to-1}\frac{S_\nu(u)}{1+u}=\infty,
\end{align*}
where we again used Lemma \ref{lem:s_transform_rel} to see that $\lim_{u\to-1}S_\nu(u)\in(0,\infty]$.
\end{proof}

Combining the above result with Theorem \ref{thm:char_r_supp_spec}, we deduce that the support of the Brown measure coincides with the spectrum for bounded elements.

\begin{corollary}
\label{cor:supp_spec}
Let $X$ be a bounded freely infinitely divisible $R$-diagonal element. Then
\begin{align*}
\supp(\mu_X)=\spec(X).
\end{align*}
\end{corollary}

The corresponding statement is false for unbounded operators. Indeed, even if the inner hole disappears, it is not necessarily true that $\supp(\mu_T)=\spec(T)$ for unbounded $T$. For example, if $\|T\|_2<\infty$ and $T$ is unbounded, then $|\lambda|>\|T\|_2$ can belong to $\spec(T)$, so by Theorem \ref{thm:char_r_supp_spec},  $\supp(\mu_T)\subsetneq\spec(T)$ can occur.

\begin{remark}
By Proposition~\ref{prop:l-2_infty}, $c^{-1}$ is not freely infinitely divisible since $\mom_{-2}(c^{-1})=\|c\|_2^2<\infty$.
\end{remark}

\subsection{Property (H)}

We formulate the regularization property (H) in our setting. A motivating example is provided by the circular Cauchy perturbation in \cite{HS2007Brown,HS2009inv}. More precisely, if $c_1$ and $c_2$ are $*$-free circular elements, then the Brown measure of $T+c_1c_2^{-1}$ has a strictly positive density on $\CC$; see \cite[Corollary~4.6]{HS2009inv}. This model was used as a regularizing device in the argument for invariant subspaces in \cite{HS2009inv}, and the same density formula is recovered in \cite[Example~4.5]{BZ2022}. The property considered below is a natural abstraction of this regularizing phenomenon, and the self-adjoint counterpart of property (H) was introduced and studied in \cite{BBgG2009regularization,HW2022Reg}.

\begin{definition}
Let $Y\in\log^+(\MM)$ be an $R$-diagonal element. We say that $Y$ satisfies property (H) if for every operator $X_0\in \log^+(\MM)$ that is $*$-free from $Y$, the Brown measure of $X_0+Y$ is absolutely continuous with respect to two-dimensional Lebesgue measure and has strictly positive density on $\CC$.
\end{definition}

To use Theorem \ref{thm:positivity}, it remains to determine when the open set $\Omega$ is all of $\CC$.

\begin{proposition}
\label{prop:omega_c}
Let $Y\in\log^+(\MM)$ be a freely infinitely divisible $R$-diagonal element such that $\tau(\ker Y)=0$ and $\mom_2(Y)=\infty$. Let $X_0\in \log^+(\MM)$ be $*$-free from $Y$. Then the Brown measure of $X_0+Y$ has strictly positive density on $\CC$.
\end{proposition}

\begin{proof}
If $X_0\in\CC$, then $X_0+Y$ is a translate of $Y$, and the statement follows from Proposition~\ref{prop:l-2_infty} and Theorem~\ref{thm:char_r_supp_spec}, so we may assume that $X_0\in\log^+(\MM)\setminus\CC$. Since $Y$ is freely infinitely divisible and $R$-diagonal, Proposition \ref{prop:l-2_infty} gives $\mom_{-2}(Y)=\infty$, and thus
\begin{align*}
F_2=\left\{\lambda\in\CC\colon \mom_2(X_0-\lambda)\le 0\right\}=\emptyset,
\end{align*}
because $X_0$ is not a scalar multiple of the identity. Moreover, since $\mom_2(Y)=\infty$, we have $F_1=\emptyset$.

Finally, since $\tau(\ker Y)=0$, we have
\begin{align*}
S=\{\lambda\in\CC\colon \tau(\ker(X_0-\lambda))\ge 1\}=\emptyset,
\end{align*}
again because $X_0$ is not scalar. Hence $\Omega=\CC$, and Theorem \ref{thm:positivity} gives the conclusion.
\end{proof}

This yields the following characterization of property (H) in our setting.

\begin{theorem}
\label{thm:property_H}
Let $Y\in\log^+(\MM)$ be a freely infinitely divisible $R$-diagonal element. Then the following are equivalent:
\begin{enumerate}
\item $Y$ satisfies property (H);
\item $\tau(\ker Y)=0$ and $\mom_2(Y)=\infty$.
\end{enumerate}
\end{theorem}

\begin{proof}
The implication (2) $\Rightarrow$ (1) follows from Proposition~\ref{prop:omega_c}. Conversely, assume that $Y$ satisfies property (H). Taking $X_0=0$, we see that $\mu_Y$ is absolutely continuous with respect to two-dimensional Lebesgue measure and has strictly positive density on $\CC$. In particular, $\mu_Y$ has no atoms and cannot have bounded support. Since $Y$ is $R$-diagonal, Theorem~\ref{thm:char_r_supp_spec} implies that $\tau(\ker Y)=0$ and $\mom_2(Y)=\infty$.
\end{proof}

\begin{example}[Integrability for $X_{m,k}$]
\label{ex:property_H_Xmk}
We return to the family $X_{m,k}$ from Example \ref{ex:x_mk}. Recall that
\begin{align*}
X_{m,k}=\left(\prod_{i=1}^m c_i\right)\left(\prod_{j=1}^k c_{m+j}\right)^{-1},\quad m\ge1,\quad k\ge0,
\end{align*}
where $(c_i)_{i\ge1}$ is a $*$-free circular family. From the description of $\mu_{|X_{m,k}|^2}$ in Example \ref{ex:x_mk} and Theorem \ref{thm:property_H}, we immediately get
\begin{align}
X_{m,k}\text{ satisfies property (H)}\quad\Longleftrightarrow\quad k\ge1.
\label{eq:property-H-Xmk}
\end{align}
Indeed, for $k\ge1$, $X_{m,k}$ has trivial kernel and the decomposition $\mu_{|X_{m,k}|^2}=f_{1/(k+1),1}\boxtimes\Pi_1^{\boxtimes m}$ forces $\mom_2(X_{m,k})=\infty$.

We record the sharper integrability threshold for $k>0$:
\begin{align}
X_{m,k}\in L^p(\MM,\tau)\quad\Longleftrightarrow\quad 0<p<\frac{2}{k+1}.
\label{eq:Lp-Xmk}
\end{align}
We follow the same strategy as in \cite[Theorem 5.4]{HS2007Brown}. For $s>0$, set $h(s)=s\tau((s^2+|X_{m,k}|^{-2})^{-1})$ and $w(s)=sh(s)$. Moreover, $w(s)=-\psi_{\mu_{|X_{m,k}|^2}}(-s^2)$, and thus the definition of the $S$-transform gives
\begin{align*}
S_{\mu_{|X_{m,k}|^2}}(-w(s))=\frac{s^2(1-w(s))}{w(s)}.
\end{align*}
Using \eqref{eq:S-Xmk}, we obtain
\begin{align}
s^2=\frac{w(s)^{k+1}}{(1-w(s))^{m+1}}.
\label{eq:w-s-Xmk}
\end{align}
Thus $w(s)$ increases from $0$ to $1$ as $s$ increases from $0$ to $\infty$, and
\begin{align*}
s=w^{(k+1)/2}(1-w)^{-(m+1)/2}.
\end{align*}
By \cite[Lemma 5.3]{HS2007Brown}, applied to $\mu_{|X_{m,k}|^{-1}}$, for $0<p<2$,
\begin{align*}
\tau(|X_{m,k}|^p)=\int_0^\infty u^{-p}\di{\mu_{|X_{m,k}|^{-1}}}(u)=\frac{2}{\pi}\sin\left(\frac{\pi p}{2}\right)\int_0^\infty s^{-p}h(s)\di{s},
\end{align*}
where the equality is understood in $[0,\infty]$. Using $h(s)=w(s)/s$ and the logarithmic derivative of \eqref{eq:w-s-Xmk}, we get
\begin{align*}
\int_0^\infty s^{-p}h(s)\di{s}&=\frac{k+1}{2}\int_0^1 w^{-(k+1)p/2}(1-w)^{(m+1)p/2}\di{w}
+\frac{m+1}{2}\int_0^1 w^{1-(k+1)p/2}(1-w)^{(m+1)p/2-1}\di{w}.
\end{align*}
The endpoint $w=1$ causes no divergence for $p>0$, while the endpoint $w=0$ is integrable precisely when $p<2/(k+1)$. Together with the elementary monotonicity of $L^p$-integrability, this proves \eqref{eq:Lp-Xmk}. In this case,
\begin{align}
\int_0^\infty s^{-p}h(s)\di{s}
=\frac{k+1}{2}B\left(1-\frac{(k+1)p}{2},1+\frac{(m+1)p}{2}\right)+\frac{m+1}{2}B\left(2-\frac{(k+1)p}{2},\frac{(m+1)p}{2}\right),
\label{eq:beta-integral-Xmk}
\end{align}
where $B$ denotes the beta function.
\end{example}

\begin{remark}
Let $z=c_1c_2^{-1}$ be the circular Cauchy element. Then $z^n$ and $X_{n,n}$ have the same $*$-distribution. Applying \eqref{eq:Lp-Xmk} with $(m,k)=(n,n)$ gives $z^n\in L^p(\MM,\tau)\Longleftrightarrow 0<p<2/(n+1)$. Thus, when $m=k=n$, the above setting reduces to that obtained in \cite[Theorem 5.4]{HS2007Brown}.
\end{remark}

\section{Regularity of the semigroup}
\label{sec:semigroup-regularity}
In the previous sections, free infinite divisibility was used mainly in a static manner. It provided structural properties of $R$-diagonal elements and allowed us to obtain support and regularity results for the corresponding Brown measures. In this section, we use it dynamically, through the free convolution semigroup associated with the symmetrized law of the modulus. More precisely, if $\mu=\widetilde{\mu}_{|Y|}$ is freely infinitely divisible, we consider the family of $R$-diagonal elements $(Y_t)_{t>0}$ determined by
\begin{align*}
\widetilde{\mu}_{|Y_t|}=\mu^{\boxplus t}.
\end{align*}
We then study the evolution in $t$ of the regularized logarithmic potential and the Fuglede--Kadison determinant of $X_0+Y_t$. This dynamical viewpoint leads to a Hamilton--Jacobi equation for the regularized logarithmic potential and yields regularity results for the Brown measures, extending phenomena previously observed in the circular and circular Cauchy cases. For this $R$-diagonal semigroup, we also derive a first-order PDE for the radial distribution function of the Brown measure of $Y_t$, closely related to the equation appearing in \cite{COR2024fractional}.

\subsection{Integrability of the semigroup}

We prove a criterion for log-integrability, which ensures that the semigroup remains log-integrable for all $t>0$. This is needed in order to work with Fuglede--Kadison determinants of possibly unbounded operators. The criterion also gives characterizations equivalent to those considered in \cite[Proposition 6.4]{BnT2002selfdecomp}. It is stated for general freely infinitely divisible probability measures, though our main application is to symmetric distributions. Once this criterion is established, it follows that the Hamiltonian appearing in the next subsection is well defined.

\begin{theorem}
\label{thm:log_integrability_semigroup}
Let $\mu$ be freely infinitely divisible, and let $(\gamma,\sigma)$ be its free generating pair, so that
\begin{align*}
\phi_\mu(z)=\gamma+\int_{\RR}\frac{1+xz}{z-x}\di{\sigma}(x),\quad z\in\CC^+.
\end{align*}
Let $T>0$. Then the following conditions are equivalent:
\begin{enumerate}
\item $\int_{\RR}\log(1+x^2)\di{\mu}(x)<\infty$.
\item $\int_{\RR}\log(1+x^2)\di{\sigma}(x)<\infty$.
\item $\int_T^\infty\frac{-\Im\phi_\mu(\iu y)}{y^2}\di{y}<\infty$.
\item $\int_0^{1/T} \Im R_\mu(\iu y)\di{y}<\infty$.
\item $\int_T^\infty\frac{\Im F_\mu(\iu y)-y}{y^2}\di{y}<\infty$.
\end{enumerate}
\end{theorem}

In particular, the functions $I_\mu$ and $\mathcal{H}_\mu$ introduced in the next subsection are well-defined.

\begin{proof}
$\mathrm{(3)}\iff\mathrm{(4)}$ is obvious. We prove $\mathrm{(2)}\iff\mathrm{(3)}$. By exchanging the integrals, we have
\begin{align*}
\int_T^\infty\frac{-\Im\phi_\mu(\iu y)}{y^2}\di{y}&=\int_T^\infty\frac{1}{y}\int_{\RR}\frac{1+x^2}{x^2+y^2}\di{\sigma}(x)\di{y}\\
&=\int_{\RR}(1+x^2)\left(\int_T^\infty\frac{1}{y(x^2+y^2)}\di{y}\right)\di{\sigma}(x)\\
&=\frac{1}{2T^2}\sigma(\{0\})+\frac{1}{2}\int_{x\neq 0}\frac{1+x^2}{x^2}\log\left(1+\frac{x^2}{T^2}\right)\di{\sigma}(x).
\end{align*}
Since $\sigma$ is finite, it follows that
\begin{align*}
\int_T^\infty\frac{-\Im\phi_\mu(\iu y)}{y^2}\di{y}<\infty\iff\int_{\RR}\log(1+x^2)\di{\sigma}(x)<\infty.
\end{align*}

We next prove $\mathrm{(3)}\iff\mathrm{(5)}$. Since $\frac{F_\mu(\iu y)}{\iu y}\to 1$ as $y\to\infty$, there exists $y_0>0$ such that for all $y\ge y_0$,
\begin{align}
|\Re F_\mu(\iu y)|\le \frac{y}{2}\le \Im F_\mu(\iu y)\le 2y,\label{eq:re_im_F_ineq}\\
\iu y=F_\mu(\iu y)+\phi_\mu(F_\mu(\iu y)).\label{eq:voiculescu_F}
\end{align}
Taking imaginary parts in (\ref{eq:voiculescu_F}), we obtain
\begin{align*}
\Im F_\mu(\iu y)-y=\Im F_\mu(\iu y)\int_{\RR}\frac{1+x^2}{(\Re F_\mu(\iu y)-x)^2+(\Im F_\mu(\iu y))^2}\di{\sigma}(x).
\end{align*}
For $y\ge y_0$, the inequalities in \eqref{eq:re_im_F_ineq} imply
\begin{align*}
c_1\frac{1+x^2}{x^2+y^2}\le\frac{1+x^2}{(\Re F_\mu(\iu y)-x)^2+(\Im F_\mu(\iu y))^2}\le c_2\frac{1+x^2}{x^2+y^2},
\end{align*}
for $x\in\RR$ and some $c_1,c_2>0$. Since $\frac{y}{2}\le \Im F_\mu(\iu y)\le 2y$, it follows that
\begin{align*}
\frac{c_1}{2}\frac{-\Im\phi_\mu(\iu y)}{y^2}\le\frac{\Im F_\mu(\iu y)-y}{y^2}\le 2c_2\frac{-\Im\phi_\mu(\iu y)}{y^2}
\end{align*}
for all $y\ge y_0$. Hence,
\begin{align*}
\int_T^\infty\frac{-\Im\phi_\mu(\iu y)}{y^2}\di{y}<\infty\iff\int_T^\infty\frac{\Im F_\mu(\iu y)-y}{y^2}\di{y}<\infty.
\end{align*}

Finally, we prove $\mathrm{(1)}\iff\mathrm{(5)}$. First, observe that
\begin{align*}
\int_{\RR}\log(1+x^2)\di{\mu}(x)=2\int_1^\infty\left(\frac{1}{y}+\Im G_\mu(\iu y)\right)\di{y}.
\end{align*}
Thus it suffices to compare $\frac{1}{y}+\Im G_\mu(\iu y)$ with $(\Im F_\mu(\iu y)-y)/y^2$. Since $G_\mu(\iu y)=1/F_\mu(\iu y)$, we have
\begin{align}
\frac{1}{y}+\Im G_\mu(\iu y)=\frac{(\Re F_\mu(\iu y))^2+\Im F_\mu(\iu y)(\Im F_\mu(\iu y)-y)}{y|F_\mu(\iu y)|^2}.\label{eq:G_F}
\end{align}
Moreover, by the relation \eqref{eq:re_im_F_ineq}, we have
\begin{align}
\frac{\Im F_\mu(\iu y)-y}{4y^2}\le\frac{\Im F_\mu(\iu y)(\Im F_\mu(\iu y)-y)}{y|F_\mu(\iu y)|^2}\le 2\frac{\Im F_\mu(\iu y)-y}{y^2}.\label{eq:F_estimates}
\end{align}
Hence, for $y\ge y_0$,
\begin{align*}
\frac{1}{y}+\Im G_\mu(\iu y)\ge\frac{\Im F_\mu(\iu y)(\Im F_\mu(\iu y)-y)}{y|F_\mu(\iu y)|^2}\ge\frac{\Im F_\mu(\iu y)-y}{4y^2}.
\end{align*}

For the reverse inequality, taking real parts in \eqref{eq:voiculescu_F}, we obtain
\begin{align*}
\Re F_\mu(\iu y)=-\gamma-\int_{\RR}\Re\frac{1+xF_\mu(\iu y)}{F_\mu(\iu y)-x}\di{\sigma}(x).
\end{align*}
Hence, by the Cauchy--Schwarz inequality,
\begin{align*}
(\Re F_\mu(\iu y))^2&\le2\gamma^2+2\sigma(\RR)\int_{\RR}\left|\frac{1+xF_\mu(\iu y)}{F_\mu(\iu y)-x}\right|^2\di{\sigma}(x)\\
&\le2\gamma^2+2\sigma(\RR)\int_{\RR}\frac{(1+x^2)(1+|F_\mu(\iu y)|^2)}{|F_\mu(\iu y)-x|^2}\di{\sigma}(x)\\
&=2\gamma^2-2\sigma(\RR)\left(\frac{1+|F_\mu(\iu y)|^2}{\Im F_\mu(\iu y)}\right)\int_{\RR}\Im\frac{1+xF_\mu(\iu y)}{F_\mu(\iu y)-x}\di{\sigma}(x)\\
&=2\gamma^2+2\sigma(\RR)\frac{1+|F_\mu(\iu y)|^2}{\Im F_\mu(\iu y)}(\Im F_\mu(\iu y)-y).
\end{align*}
Using again the bounds in \eqref{eq:re_im_F_ineq}, it follows that
\begin{align*}
\frac{(\Re F_\mu(\iu y))^2}{y|F_\mu(\iu y)|^2}\le c_3\left(\frac{1}{y^3}+\frac{\Im F_\mu(\iu y)-y}{y^2}\right)
\end{align*}
for some constant $c_3>0$ and all $y\ge y_0$. Together with \eqref{eq:G_F} and \eqref{eq:F_estimates}, this yields
\begin{align*}
\frac{1}{y}+\Im G_\mu(\iu y)\le\frac{c_4}{y^3}+c_5\frac{\Im F_\mu(\iu y)-y}{y^2}
\end{align*}
for some constants $c_4,c_5>0$ and all $y\ge y_0$. Hence, we conclude that
\begin{align*}
\int_1^\infty\left(\frac{1}{y}+\Im G_\mu(\iu y)\right)\di{y}<\infty\iff\int_1^\infty\frac{\Im F_\mu(\iu y)-y}{y^2}\di{y}<\infty.
\end{align*}
This proves $\mathrm{(1)}\iff\mathrm{(5)}$ and completes the proof.
\end{proof}

\begin{corollary}
Let $\mu$ be freely infinitely divisible, and let $(\gamma,\sigma)$ be its free generating pair. Then the following are equivalent:
\begin{enumerate}
\item $\int_{\RR}\log(1+x^2)\di{\mu^{\boxplus t}}(x)<\infty$ for some $t>0$.
\item $\int_{\RR}\log(1+x^2)\di{\mu^{\boxplus t}}(x)<\infty$ for every $t>0$.
\item $\int_{\RR}\log(1+x^2)\di{\sigma}(x)<\infty$.
\end{enumerate}
In particular, if $\mu$ is log-integrable, then the whole free convolution semigroup $(\mu^{\boxplus t})_{t>0}$ is log-integrable.
\end{corollary}

\begin{proof}
The free generating pair of $\mu^{\boxplus t}$ is $(t\gamma,t\sigma)$. Hence the conclusion follows directly from Theorem \ref{thm:log_integrability_semigroup}.
\end{proof}

\subsection{Determinant regularity}

We study the Fuglede--Kadison determinant along the semigroup generated by a freely infinitely divisible $R$-diagonal perturbation. Let $X_0\in\log^+(\MM)$ and let $(Y_t)_{t\ge0}$ be a family of $R$-diagonal elements such that $X_0$ and $Y_t$ are $*$-free for every $t\ge0$ and
\begin{align*}
\widetilde{\mu}_{|Y_t|}=\mu^{\boxplus t},\quad t\ge0,
\end{align*}
where $\mu=\widetilde{\mu}_{|Y_1|}$ is a symmetric freely infinitely divisible measure and $Y_1\in\log^+(\MM)$. It follows from Theorem \ref{thm:log_integrability_semigroup} that $Y_t\in\log^+(\MM)$ for all $t>0$. We write $X_t\coloneqq X_0+Y_t$. This framework includes several models previously considered in the literature. For instance, when $\mu$ is the semicircle distribution, the model corresponds to the circular flow studied in \cite{HZ2023Brown,BYZ2024Brown,Z2026brown}. Furthermore, the case where $\mu$ is the Cauchy distribution corresponds to the circular Cauchy perturbation treated in \cite{HS2009inv}.

\subsubsection{Global inversion formula in our setting}

Fix $\lambda\in \CC$ and $t>0$. Set
\begin{align*}
\mu_{0,\lambda}\coloneqq \widetilde{\mu}_{|X_0-\lambda|},\quad\mu_t\coloneqq \widetilde{\mu}_{|Y_t|}=\mu^{\boxplus t}.
\end{align*}
Define
\begin{align*}
H_{\lambda,t}(z)\coloneqq z+\phi_{\mu_t}(F_{\mu_{0,\lambda}}(z))=z+tR_\mu(G_{\mu_{0,\lambda}}(z)).
\end{align*}
By Theorem \ref{thm:global_inversion}, there exists a domain $\Omega_{\lambda,t}\subseteq\CC^+$ such that $H_{\lambda,t}\colon\Omega_{\lambda,t}\to\CC^+$ is a conformal isomorphism. We denote this inverse, which is the first subordination function, by $\omega_{1,t}^{(\lambda)}\colon\CC^+\to\Omega_{\lambda,t}$. Let $\omega_{2,t}^{(\lambda)}$ denote the corresponding second subordination function.

For a symmetric probability measure $\nu$ on $\RR$, we write
\begin{align*}
\Gsf_\nu(y)\coloneqq -\Im G_\nu(\iu y),\quad y>0.
\end{align*}
We also set
\begin{align*}
\Gsf_{t,\lambda}(y)\coloneqq -\Im G_{\widetilde{\mu}_{|X_t-\lambda|}}(\iu y),\quad\Rsf_\mu(y)\coloneqq \Im R_\mu(\iu y),\quad y>0.
\end{align*}
By symmetry, $\omega_{j,t}^{(\lambda)}(\iu\ee)$ is purely imaginary for every $\ee>0$ and $j=1,2$. We write
\begin{align*}
\omega_{j,t}^{(\lambda)}(\iu\ee)=\iu W_{j,t}^{(\lambda)}(\ee),\quad \ee>0,\ j=1,2.
\end{align*}
By the subordination relation \eqref{eq:rdiag-subordination-F} and the identity \eqref{eq:symm-modulus-cauchy}, we have
\begin{align}
\Gsf_{t,\lambda}(\ee)&=\Gsf_{\mu_{0,\lambda}}(W_{1,t}^{(\lambda)}(\ee))=W_{1,t}^{(\lambda)}(\ee)\tau((|X_0-\lambda|^2+(W_{1,t}^{(\lambda)}(\ee))^2)^{-1}), \label{eq:Gsf-X0-resolvent}\\
\Gsf_{t,\lambda}(\ee)&=\Gsf_{\mu_t}(W_{2,t}^{(\lambda)}(\ee))=W_{2,t}^{(\lambda)}(\ee)\tau((|Y_t|^2+(W_{2,t}^{(\lambda)}(\ee))^2)^{-1}).
\label{eq:Gsf-Yt-resolvent}
\end{align}
Moreover, passing to imaginary parts in \eqref{eq:rdiag-subordination-sum} gives
\begin{align}
\Gsf_{t,\lambda}(\ee)^{-1}=W_{1,t}^{(\lambda)}(\ee)+W_{2,t}^{(\lambda)}(\ee)-\ee.
\label{eq:Gsf-inverse-W}
\end{align}
The identity $H_{\lambda,t}(\omega_{1,t}^{(\lambda)}(\iu\ee))=\iu\ee$ gives
\begin{align}
W_{1,t}^{(\lambda)}(\ee)-\ee&=t\Rsf_\mu(\Gsf_{\mu_{0,\lambda}}(W_{1,t}^{(\lambda)}(\ee)))=t\Rsf_\mu(\Gsf_{t,\lambda}(\ee)).
\label{eqn:w-fixed-pt}
\end{align}

If $W_{1,t}^{(\lambda)}(0)>0$, then passing to the limit $\ee\to 0$ in the preceding identity yields
\begin{align*}
W_{1,t}^{(\lambda)}(0)=t\Rsf_\mu(\Gsf_{\mu_{0,\lambda}}(W_{1,t}^{(\lambda)}(0))).
\end{align*}

\subsubsection{Hamilton--Jacobi equation}

\begin{theorem}
\label{thm:determinantal-regularity}
With the notation above, the following hold:
\begin{enumerate}
\item For every $t>0$, $\ee>0$, and $\lambda\in \CC$,
\begin{align*}
\tau(\log(|X_t-\lambda|^2+\ee^2))=\tau(\log(|X_0-\lambda|^2+(W_{1,t}^{(\lambda)}(\ee))^2))+tI_\mu(\Gsf_{t,\lambda}(\ee)),
\end{align*}
where 
\begin{align*}
    I_\mu(y)&=2\int_0^y\Rsf_\mu(r)\di{r}-2y\Rsf_\mu(y),\quad y\ge0.
\end{align*}
\item For every fixed $\lambda\in \CC$, the function $t\mapsto \Delta(X_t-\lambda)$ is increasing on $[0,\infty)$. More precisely, for every $\ee>0$,
\begin{align*}
\partial_t\tau(\log(|X_t-\lambda|^2+\ee^2))=2\int_0^{\Gsf_{t,\lambda}(\ee)}\Rsf_\mu(r)\di{r}\ge 0.
\end{align*}
Moreover, $\lim_{t\to 0}\Delta(X_t-\lambda)=\Delta(X_0-\lambda)$.
\item As $t\to 0$, the Brown measure $\mu_{X_t}$ converges weakly to $\mu_{X_0}$.
\end{enumerate}
\end{theorem}

\begin{proof}
Fix $\lambda\in\CC$, $t>0$, and $\ee>0$. By Proposition \ref{prop:log_symm_subord} and \eqref{eq:Gsf-inverse-W}, we have
\begin{align}
\tau(\log(|X_t-\lambda|^2+\ee^2))&=\tau(\log(|X_0-\lambda|^2+(W_{1,t}^{(\lambda)}(\ee))^2))+C_{t,\lambda}(\ee),
\label{eq:determinant-decomposition}
\end{align}
where
\begin{align*}
C_{t,\lambda}(\ee)\coloneqq \tau(\log(|Y_t|^2+(W_{2,t}^{(\lambda)}(\ee))^2))+2\log\Gsf_{t,\lambda}(\ee).
\end{align*}

We first identify $C_{t,\lambda}(\ee)$. Differentiating \eqref{eq:Gsf-inverse-W} with respect to $\ee$, we obtain
\begin{align}
\frac{2\partial_\ee\Gsf_{t,\lambda}(\ee)}{\Gsf_{t,\lambda}(\ee)}&=-2\Gsf_{t,\lambda}(\ee)(\partial_\ee W_{1,t}^{(\lambda)}(\ee)+\partial_\ee W_{2,t}^{(\lambda)}(\ee)-1).
\label{eq:epsilon-derivative-Gsf-inverse}
\end{align}
On the other hand, differentiating \eqref{eqn:w-fixed-pt} with respect to $\ee$, we have
\begin{align}
\partial_\ee W_{1,t}^{(\lambda)}(\ee)-1&=t\Rsf_\mu^\prime(\Gsf_{t,\lambda}(\ee))\partial_\ee\Gsf_{t,\lambda}(\ee).
\label{eq:epsilon-derivative-fixed-point}
\end{align}
By \eqref{eq:Gsf-Yt-resolvent},
\begin{align*}
\partial_\ee\tau(\log(|Y_t|^2+(W_{2,t}^{(\lambda)}(\ee))^2))=2\Gsf_{t,\lambda}(\ee)\partial_\ee W_{2,t}^{(\lambda)}(\ee).
\end{align*}
Therefore, using \eqref{eq:epsilon-derivative-Gsf-inverse} and \eqref{eq:epsilon-derivative-fixed-point}, we get
\begin{align*}
\partial_\ee C_{t,\lambda}(\ee)&=2\Gsf_{t,\lambda}(\ee)\partial_\ee W_{2,t}^{(\lambda)}(\ee)+\frac{2\partial_\ee\Gsf_{t,\lambda}(\ee)}{\Gsf_{t,\lambda}(\ee)}\\
&=-2\Gsf_{t,\lambda}(\ee)(\partial_\ee W_{1,t}^{(\lambda)}(\ee)-1)\\
&=-2t\Gsf_{t,\lambda}(\ee)\Rsf_\mu^\prime(\Gsf_{t,\lambda}(\ee))\partial_\ee\Gsf_{t,\lambda}(\ee).
\end{align*}
Since $I_\mu^\prime(y)=-2y\Rsf_\mu^\prime(y)$, it follows that
\begin{align}
\partial_\ee\left(C_{t,\lambda}(\ee)-tI_\mu(\Gsf_{t,\lambda}(\ee))\right)&=0.
\label{eq:C-minus-tI-constant}
\end{align}
Thus $C_{t,\lambda}(\ee)-tI_\mu(\Gsf_{t,\lambda}(\ee))$ is independent of $\ee$.

We now show that this constant is zero. As $\ee\to\infty$, one has $\Gsf_{t,\lambda}(\ee)\to0$ and $W_{2,t}^{(\lambda)}(\ee)\to\infty$. Moreover, by \eqref{eq:Gsf-inverse-W} and \eqref{eqn:w-fixed-pt},
\begin{align}
\Gsf_{t,\lambda}(\ee)W_{2,t}^{(\lambda)}(\ee)&=1-t\Gsf_{t,\lambda}(\ee)\Rsf_\mu(\Gsf_{t,\lambda}(\ee)).
\label{eq:Gsf-W2-asymptotic}
\end{align}
Since $y\Rsf_\mu(y)\to0$ as $y\to0$, \eqref{eq:Gsf-W2-asymptotic} implies $\Gsf_{t,\lambda}(\ee)W_{2,t}^{(\lambda)}(\ee)\to1$. Hence
\begin{align*}
\tau(\log(|Y_t|^2+(W_{2,t}^{(\lambda)}(\ee))^2))-2\log W_{2,t}^{(\lambda)}(\ee)\to0
\end{align*}
by dominated convergence, and $2\log(\Gsf_{t,\lambda}(\ee)W_{2,t}^{(\lambda)}(\ee))\to0$. Therefore $C_{t,\lambda}(\ee)\to0$. Since $I_\mu(0)=0$ and $\Gsf_{t,\lambda}(\ee)\to0$, we also have $tI_\mu(\Gsf_{t,\lambda}(\ee))\to0$. Thus the constant in \eqref{eq:C-minus-tI-constant} is zero, so
\begin{align}
C_{t,\lambda}(\ee)&=tI_\mu(\Gsf_{t,\lambda}(\ee)).
\label{eq:C-equals-tI}
\end{align}
Combining \eqref{eq:determinant-decomposition} and \eqref{eq:C-equals-tI} proves the identity in $\mathrm{(1)}$.

We next prove $\mathrm{(2)}$. For each fixed $\ee>0$, \eqref{eqn:w-fixed-pt} implies $W_{1,t}^{(\lambda)}(\ee)\to\ee$ as $t\to0$. Differentiating \eqref{eqn:w-fixed-pt} with respect to $t$, we have
\begin{align}
\partial_tW_{1,t}^{(\lambda)}(\ee)&=\Rsf_\mu(\Gsf_{t,\lambda}(\ee))+t\Rsf_\mu^\prime(\Gsf_{t,\lambda}(\ee))\partial_t\Gsf_{t,\lambda}(\ee).
\label{eq:t-derivative-fixed-point}
\end{align}
Using \eqref{eq:Gsf-X0-resolvent}, we also have
\begin{align*}
\partial_t\Gsf_{t,\lambda}(\ee)&=\partial_tW_{1,t}^{(\lambda)}(\ee)\tau((|X_0-\lambda|^2+(W_{1,t}^{(\lambda)}(\ee))^2)^{-1})\\
&\quad-2(W_{1,t}^{(\lambda)}(\ee))^2\partial_tW_{1,t}^{(\lambda)}(\ee)\tau((|X_0-\lambda|^2+(W_{1,t}^{(\lambda)}(\ee))^2)^{-2}).
\end{align*}
Differentiating the identity in $\mathrm{(1)}$ with respect to $t$, and then using \eqref{eq:t-derivative-fixed-point} together with $I_\mu^\prime(y)=-2y\Rsf_\mu^\prime(y)$, we obtain
\begin{align*}
\partial_t\tau(\log(|X_t-\lambda|^2+\ee^2))&=2\Gsf_{t,\lambda}(\ee)\partial_tW_{1,t}^{(\lambda)}(\ee)+I_\mu(\Gsf_{t,\lambda}(\ee))+tI_\mu^\prime(\Gsf_{t,\lambda}(\ee))\partial_t\Gsf_{t,\lambda}(\ee)\\
&=2\Gsf_{t,\lambda}(\ee)\Rsf_\mu(\Gsf_{t,\lambda}(\ee))+I_\mu(\Gsf_{t,\lambda}(\ee))\\
&=2\int_0^{\Gsf_{t,\lambda}(\ee)}\Rsf_\mu(r)\di{r}\ge0.
\end{align*}
Thus the regularized Fuglede--Kadison determinant $\Delta((|X_t-\lambda|^2+\ee^2)^{1/2})$ is increasing in $t$ for every $\ee>0$. Letting $\ee\to0$ gives the monotonicity of $t\mapsto\Delta(X_t-\lambda)$.

It remains to show that $\lim_{t\to0}\Delta(X_t-\lambda)=\Delta(X_0-\lambda)$. For fixed $\ee>0$, the identity in $\mathrm{(1)}$, together with $W_{1,t}^{(\lambda)}(\ee)\to\ee$ and $tI_\mu(\Gsf_{t,\lambda}(\ee))\to0$, gives
\begin{align*}
\tau(\log(|X_t-\lambda|^2+\ee^2))\to\tau(\log(|X_0-\lambda|^2+\ee^2))
\end{align*}
as $t\to0$. Hence, for every $\ee>0$,
\begin{align*}
\limsup_{t\to0}\log\Delta(X_t-\lambda)\le\frac{1}{2}\tau(\log(|X_0-\lambda|^2+\ee^2)).
\end{align*}
Letting $\ee\to0$ gives
\begin{align*}
\limsup_{t\to0}\log\Delta(X_t-\lambda)\le\log\Delta(X_0-\lambda).
\end{align*}
The reverse inequality follows from the monotonicity already proved. Hence $\lim_{t\to0}\Delta(X_t-\lambda)=\Delta(X_0-\lambda)$.

Finally, we prove $\mathrm{(3)}$. Define $u_t(\lambda)\coloneqq\log\Delta(X_t-\lambda)$ and $u_0(\lambda)\coloneqq\log\Delta(X_0-\lambda)$. Each $u_t$ is subharmonic. Fix $\phi\in C_c^\infty(\CC)$ and $t_0>0$. By $\mathrm{(2)}$, we have $u_0(\lambda)\le u_t(\lambda)\le u_{t_0}(\lambda)$ for every $t\in[0,t_0]$, and $u_t(\lambda)\to u_0(\lambda)$ for every $\lambda\in\CC$ as $t\to0$. Since subharmonic functions are locally integrable and $|u_t(\lambda)|\le |u_0(\lambda)|+|u_{t_0}(\lambda)|$, the dominated convergence theorem implies
\begin{align*}
\int_{\CC}\phi(\lambda)\di{\mu_{X_t}}(\lambda)&=\frac{1}{2\pi}\int_{\CC}(\Delta\phi)(\lambda)u_t(\lambda)\di{\lambda_1}\di{\lambda_2}\to \frac{1}{2\pi}\int_{\CC}(\Delta\phi)(\lambda)u_0(\lambda)\di{\lambda_1}\di{\lambda_2}=\int_{\CC}\phi(\lambda)\di{\mu_{X_0}}(\lambda).
\end{align*}
Thus $\mu_{X_t}\to\mu_{X_0}$ weakly as $t\to0$.
\end{proof}

As an immediate consequence of Theorem \ref{thm:determinantal-regularity}, we obtain the following Hamilton--Jacobi equation. This type of equation also appears in related models, such as \cite{DHK2022free,HZ2023Brown,HH2023brown}; in particular, it extends the circular case studied in \cite[Proposition 3.2]{HZ2023Brown}.

\begin{proposition}
\label{prop:hj_pde}
For $\lambda\in \CC$, define
\begin{align*}
S(t,\lambda,\ee)\coloneqq \tau(\log(|X_t-\lambda|^2+\ee^2)).
\end{align*}
Then $S(t,\lambda,\ee)$ satisfies the following Hamilton--Jacobi equation:
\begin{align*}
\partial_tS(t,\lambda,\ee)=2\mathcal H_\mu\left(\frac{\partial_\ee S(t,\lambda,\ee)}{2}\right),
\end{align*}
where $\mathcal H_\mu$ is defined by
\begin{align*}
\mathcal H_\mu(y)\coloneqq \int_0^y\Rsf_\mu(r)\di{r},\quad y\ge 0.
\end{align*}
\end{proposition}

\begin{proof}
By definition, $\partial_\ee S(t,\lambda,\ee)=2\Gsf_{t,\lambda}(\ee)$. The identity then follows from Theorem \ref{thm:determinantal-regularity}(2).
\end{proof}

\begin{example}
We first consider the free circular Brownian motion model $X_0+c_t$, where $c_t$ is a circular element with variance $t$. Let $c$ denote the circular element at time $t=1$. Then the symmetrized law $\mu_0\coloneqq\widetilde{\mu}_{|c|}$ of the radial part is the centered semicircle law, so $\Rsf_{\mu_0}(r)=r$. Proposition \ref{prop:hj_pde} gives
\begin{align*}
\partial_t S(t,\lambda,\ee)=\frac{1}{4}(\partial_\ee S(t,\lambda,\ee))^2.
\end{align*}
This relation was established in \cite[Proposition 3.2]{HZ2023Brown}, where it was used to analyze the regularity of the Brown measure; see also \cite{H2021pde} for an overview of PDE methods.

Proposition \ref{prop:hj_pde} also applies beyond the circular flow. Let $c_1c_2^{-1}$ be the circular Cauchy element, and consider the perturbation model $X_0+t c_1c_2^{-1}$. This model and its properties were extensively studied in \cite{HS2007Brown,HS2009inv}. The symmetrized law $\mu_1\coloneqq\widetilde{\mu}_{|c_1c_2^{-1}|}$ is the standard Cauchy law, and therefore $\Rsf_{\mu_1}(r)=1$. Thus the Hamilton--Jacobi equation reduces to
\begin{align*}
\partial_t S(t,\lambda,\ee)=\partial_\ee S(t,\lambda,\ee).
\end{align*}

More generally, consider the subfamily $X_{1,k}$ from Example \ref{ex:x_mk}. Recall that $\mu_k\coloneqq\widetilde{\mu}_{|X_{1,k}|}$ is the symmetric free stable law of index $2/(k+1)$, and its $R$-transform satisfies $\Rsf_{\mu_k}(r)=r^{(1-k)/(k+1)}$. Hence we obtain
\begin{align*}
\partial_t S(t,\lambda,\ee)=\frac{k+1}{2^{2/(k+1)}}(\partial_\ee S(t,\lambda,\ee))^{2/(k+1)}.
\end{align*}
\end{example}

\subsection{A PDE for the radial distribution function}
\label{subsec:radial_CDF_equation}

Proposition \ref{prop:radial-voiculescu} can also be applied to this free convolution semigroup. This gives a first-order partial differential equation for the radial distribution function.

\begin{proposition}
\label{prop:radial_CDF_equation}
Let $\mu=\widetilde{\mu}_{|Y_1|}$ and let $(0,\Lambda)$ be the free generating pair of $\mu$. For $t>0$ and $q>0$, set
\begin{align}
\theta_t(q)=1+\frac{t}{\iu q}\phi_\mu(\iu q)=1-t\int_{\RR}\frac{1+s^2}{q^2+s^2}\di{\Lambda}(s).
\label{eq:theta-semigroup}
\end{align}
Then there exists an open interval $I_t\subseteq\RR_{>0}$ such that $\theta_t\colon I_t\to(\mu^{\boxplus t}(\{0\}),1)$ is an increasing bijection. Moreover, for $q\in I_t$,
\begin{align}
\mu_{Y_t}(\{z\in\CC \colon |z|\le r_t(q)\})=\theta_t(q),\quad r_t(q)^2=q^2\theta_t(q)(1-\theta_t(q)).
\label{eq:radial-parametrization-semigroup}
\end{align}
Let $F(t,r)=\mu_{Y_t}(\{z\in\CC \colon |z|\le r\})$. Then, for $0<r<\mom_2(Y_t)^{1/2}$, $F$ satisfies
\begin{align}
t\partial_tF(t,r)+\frac{r(2F(t,r)-1)}{2F(t,r)}\partial_rF(t,r)=-1+F(t,r).
\label{eq:unscaled-radial-CDF-PDE}
\end{align}
Equivalently, if $\widehat F(t,r)=F(t,tr)$, then
\begin{align*}
t\partial_t\widehat F(t,r)=\frac{r\partial_r\widehat F(t,r)}{2\widehat F(t,r)}-1+\widehat F(t,r).
\end{align*}
\end{proposition}

The last equation is analogous to the one for the radial distribution function of the Brown measure in \cite{COR2024fractional}.

\begin{proof}
The first part follows from Proposition \ref{prop:radial-voiculescu} applied to $Y_t$. By Proposition \ref{prop:l-2_infty}, we have $\mom_{-2}(Y_t)=\infty$. Hence the inner radius of the Brown measure of $Y_t$ is zero. Together with Theorem \ref{thm:char_r_supp_spec} and Proposition \ref{prop:radial-voiculescu}, this shows that $q\mapsto r_t(q)$ parametrizes the interior radial support $0<r<\mom_2(Y_t)^{1/2}$. Thus, on this region, we may write $F(t,r_t(q))=\theta_t(q)$. For fixed $q$, \eqref{eq:theta-semigroup} gives
\begin{align*}
\partial_t\theta_t(q)=-\int_{\RR}\frac{1+s^2}{q^2+s^2}\di{\Lambda}(s)=-\frac{1-\theta_t(q)}{t}.
\end{align*}
On the other hand, differentiating $r_t(q)^2=q^2\theta_t(q)(1-\theta_t(q))$ with respect to $t$, we obtain
\begin{align*}
\partial_t r_t(q)=\frac{r_t(q)(2\theta_t(q)-1)}{2t\theta_t(q)}.
\end{align*}
The chain rule applied to $F(t,r_t(q))=\theta_t(q)$ yields
\begin{align*}
\partial_tF(t,r_t(q))+\frac{r_t(q)(2\theta_t(q)-1)}{2t\theta_t(q)}\partial_rF(t,r_t(q))=-\frac{1-\theta_t(q)}{t}.
\end{align*}
Substituting $\theta_t(q)=F(t,r_t(q))$, we have \eqref{eq:unscaled-radial-CDF-PDE}. The last equation follows immediately.
\end{proof}

\section*{Acknowledgments}
We would like to thank Takahiro Hasebe, Hao-Wei Huang, Noriyoshi Sakuma, and Jiun-Chau Wang for their helpful comments and suggestions.
P. Zhong is supported in part by NSF grant LEAPS-MPS-2516951 and NSF CAREER Award DMS-2516950.

\printbibliography

@misc{BZ2022,
 author = {Bercovici, Hari and Zhong, Ping},
 title = {The {Brown} measure of a sum of two free nonselfadjoint random variables, one of which is {R}-diagonal},
 year = {2022},
 howpublished = {Preprint, {arXiv}:2209.12379 [math.{PR}] (2022)},
 url = {https://arxiv.org/abs/2209.12379},
 arXiv = {arXiv:2209.12379}
}

@article{BNNS2018eta,
 author = {Bercovici, Hari and Nica, Alexandru and Noyes, Michael and Szpojankowski, Kamil},
 title = {Eta-diagonal distributions and infinite divisibility for {R}-diagonals},
 fjournal = {Annales de l'Institut Henri Poincar{\'e}. Probabilit{\'e}s et Statistiques},
 journal = {Ann. Inst. Henri Poincar{\'e}, Probab. Stat.},
 issn = {0246-0203},
 volume = {54},
 number = {2},
 pages = {907--937},
 year = {2018},
 doi = {10.1214/17-AIHP826},
 zbMATH = {6897973},
 Zbl = {1400.46052}
}

@article{AHS2013law,
 author = {Arizmendi, Octavio and Hasebe, Takahiro and Sakuma, Noriyoshi},
 title = {On the law of free subordinators},
 fjournal = {ALEA. Latin American Journal of Probability and Mathematical Statistics},
 journal = {ALEA, Lat. Am. J. Probab. Math. Stat.},
 issn = {1980-0436},
 volume = {10},
 number = {1},
 pages = {271--291},
 year = {2013},
 url = {alea.impa.br/articles/v10/10-12.pdf},
 zbMATH = {6236550},
 Zbl = {1291.46060}
}

@article{COR2024fractional,
 author = {Campbell, Andrew and O'Rourke, Sean and Renfrew, David},
 title = {The fractional free convolution of {{\(R\)}}-diagonal elements and random polynomials under repeated differentiation},
 fjournal = {IMRN. International Mathematics Research Notices},
 journal = {Int. Math. Res. Not.},
 issn = {1073-7928},
 volume = {2024},
 number = {13},
 pages = {10189--10218},
 year = {2024},
 doi = {10.1093/imrn/rnae062},
 zbMATH = {7935859},
 Zbl = {1559.60012}
}

@article{HW2022Reg,
 author = {Huang, Hao-Wei and Wang, Jiun-Chau},
 title = {Regularity results for free {L{\'e}vy} processes},
 fjournal = {Advances in Mathematics},
 journal = {Adv. Math.},
 issn = {0001-8708},
 volume = {402},
 pages = {42},
 note = {Id/No 108323},
 year = {2022},
 doi = {10.1016/j.aim.2022.108323},
 zbMATH = {7524880},
 Zbl = {1497.46074}
}

@article{HS2007Brown,
 author = {Haagerup, Uffe and Schultz, Hanne},
 title = {Brown measures of unbounded operators affiliated with a finite von {Neumann} algebra},
 fjournal = {Mathematica Scandinavica},
 journal = {Math. Scand.},
 issn = {0025-5521},
 volume = {100},
 number = {2},
 pages = {209--263},
 year = {2007},
 doi = {10.7146/math.scand.a-15023},
 zbMATH = {5361525},
 Zbl = {1168.46039}
}

@article{HS2009inv,
 author = {Haagerup, Uffe and Schultz, Hanne},
 title = {Invariant subspaces for operators in a general {{\(\text{II}_{1}\)}}-factor},
 fjournal = {Publications Math{\'e}matiques},
 journal = {Publ. Math., Inst. Hautes {\'E}tud. Sci.},
 issn = {0073-8301},
 volume = {109},
 pages = {19--111},
 year = {2009},
 doi = {10.1007/s10240-009-0018-7},
 zbMATH = {5570772},
 Zbl = {1178.46058}
}

@incollection{HM2013Law,
 author = {Haagerup, Uffe and M{\"o}ller, S{\"o}ren},
 title = {The law of large numbers for the free multiplicative convolution},
 booktitle = {Operator algebra and dynamics. Nordforsk network closing conference, Gj\'ogv, Faroe Islands, May 2012},
 isbn = {978-3-642-39458-4; 978-3-642-39459-1},
 pages = {157--186},
 year = {2013},
 publisher = {Berlin: Springer},
 doi = {10.1007/978-3-642-39459-1_8},
 zbMATH = {6420320},
 Zbl = {1319.46045}
}

@article{HL2000Brown,
 author = {Haagerup, Uffe and Larsen, Flemming},
 title = {Brown's spectral distribution measure for {{\(R\)}}-diagonal elements in finite von {Neumann} algebras},
 fjournal = {Journal of Functional Analysis},
 journal = {J. Funct. Anal.},
 issn = {0022-1236},
 volume = {176},
 number = {2},
 pages = {331--367},
 year = {2000},
 doi = {10.1006/jfan.2000.3610},
 zbMATH = {1529988},
 Zbl = {0984.46042}
}

@article{BYZ2024Brown,
 author = {Belinschi, Serban and Yin, Zhi and Zhong, Ping},
 title = {The {Brown} measure of a sum of two free random variables, one of which is triangular elliptic},
 fjournal = {Advances in Mathematics},
 journal = {Adv. Math.},
 issn = {0001-8708},
 volume = {441},
 pages = {62},
 note = {Id/No 109562},
 year = {2024},
 doi = {10.1016/j.aim.2024.109562},
 zbMATH = {7823132},
 Zbl = {1547.46060}
}

@article{BnT2002selfdecomp,
 author = {Barndorff-Nielsen, Ole E. and Thorbj{\o}rnsen, Steen},
 title = {Self-decomposability and {L{\'e}vy} processes in free probability},
 fjournal = {Bernoulli},
 journal = {Bernoulli},
 issn = {1350-7265},
 volume = {8},
 number = {3},
 pages = {323--366},
 year = {2002},
 zbMATH = {1824601},
 Zbl = {1024.60022}
}

@article{B1998process,
 author = {Biane, Philippe},
 title = {Processes with free increments},
 fjournal = {Mathematische Zeitschrift},
 journal = {Math. Z.},
 issn = {0025-5874},
 volume = {227},
 number = {1},
 pages = {143--174},
 year = {1998},
 doi = {10.1007/PL00004363},
 zbMATH = {1111622},
 Zbl = {0902.60060}
}

@article{V1993analogue,
 author = {Voiculescu, Dan},
 title = {The analogues of entropy and of {Fisher}'s information measure in free probability theory. {I}},
 fjournal = {Communications in Mathematical Physics},
 journal = {Commun. Math. Phys.},
 issn = {0010-3616},
 volume = {155},
 number = {1},
 pages = {71--92},
 year = {1993},
 doi = {10.1007/BF02100050},
 zbMATH = {459089},
 Zbl = {0781.60006}
}

@article{BB2007new,
 author = {Belinschi, S. T. and Bercovici, H.},
 title = {A new approach to subordination results in free probability},
 fjournal = {Journal d'Analyse Math{\'e}matique},
 journal = {J. Anal. Math.},
 issn = {0021-7670},
 volume = {101},
 pages = {357--365},
 year = {2007},
 doi = {10.1007/s11854-007-0013-1},
 zbMATH = {5243808},
 Zbl = {1142.46030}
}

@incollection{NS1997rdiag,
 author = {Nica, Alexandru and Speicher, Roland},
 title = {{{\(R\)}}-diagonal pairs. -- {A} common approach to {Haar} unitaries and circular elements},
 booktitle = {Free probability theory. Papers from a workshop on random matrices and operator algebra free products, Toronto, Canada, Mars 1995},
 isbn = {0-8218-0675-0},
 pages = {149--188},
 year = {1997},
 publisher = {Providence, RI: American Mathematical Society},
 zbMATH = {1003153},
 Zbl = {0889.46053}
}

@article {BB2004,
    AUTHOR = {Belinschi, S. T. and Bercovici, H.},
     TITLE = {Atoms and regularity for measures in a partially defined free
              convolution semigroup},
   JOURNAL = {Math. Z.},
  FJOURNAL = {Mathematische Zeitschrift},
    VOLUME = {248},
      YEAR = {2004},
    NUMBER = {4},
     PAGES = {665--674},
      ISSN = {0025-5874,1432-1823},
   MRCLASS = {46L54 (30B40 30E20 31A15)},
  MRNUMBER = {2103535},
MRREVIEWER = {Rolf\ Gohm},
       DOI = {10.1007/s00209-004-0671-y},
       URL = {https://doi-org.ezproxy.lib.uh.edu/10.1007/s00209-004-0671-y},
}

@article{BV1993free,
 author = {Bercovici, Hari and Voiculescu, Dan},
 title = {Free convolution of measures with unbounded support},
 fjournal = {Indiana University Mathematics Journal},
 journal = {Indiana Univ. Math. J.},
 issn = {0022-2518},
 volume = {42},
 number = {3},
 pages = {733--773},
 year = {1993},
 doi = {10.1512/iumj.1993.42.42033},
 zbMATH = {527331},
 Zbl = {0806.46070}
}

@article{BB2005partially,
 author = {Belinschi, S. T. and Bercovici, H.},
 title = {Partially defined semigroups relative to multiplicative free convolution},
 fjournal = {IMRN. International Mathematics Research Notices},
 journal = {Int. Math. Res. Not.},
 issn = {1073-7928},
 volume = {2005},
 number = {2},
 pages = {65--101},
 year = {2005},
 doi = {10.1155/IMRN.2005.65},
 zbMATH = {2162780},
 Zbl = {1092.46046}
}

@article{Z2026brown,
 author = {Zhong, Ping},
 title = {Brown measure of the sum of an elliptic operator and a free random variable in a finite von {Neumann} algebra},
 fjournal = {American Journal of Mathematics},
 journal = {Am. J. Math.},
 issn = {0002-9327},
 volume = {148},
 number = {1},
 pages = {161--233},
 year = {2026},
 doi = {10.1353/ajm.2026.a980771},
 zbMATH = {8159499}
}

@article{DHK2022free,
 author = {Driver, Bruce K. and Hall, Brian and Kemp, Todd},
 title = {The {Brown} measure of the free multiplicative {Brownian} motion},
 fjournal = {Probability Theory and Related Fields},
 journal = {Probab. Theory Relat. Fields},
 issn = {0178-8051},
 volume = {184},
 number = {1-2},
 pages = {209--273},
 year = {2022},
 doi = {10.1007/s00440-022-01142-z},
 zbMATH = {7606003},
 Zbl = {1500.60053}
}

@article{HZ2023Brown,
 author = {Ho, Ching-Wei and Zhong, Ping},
 title = {Brown measures of free circular and multiplicative {Brownian} motions with self-adjoint and unitary initial conditions},
 fjournal = {Journal of the European Mathematical Society (JEMS)},
 journal = {J. Eur. Math. Soc. (JEMS)},
 issn = {1435-9855},
 volume = {25},
 number = {6},
 pages = {2163--2227},
 year = {2023},
 doi = {10.4171/JEMS/1233},
 zbMATH = {7714610},
 Zbl = {1529.46047}
}

@article{HH2023brown,
 author = {Hall, Brian C. and Ho, Ching-Wei},
 title = {The {Brown} measure of a family of free multiplicative {Brownian} motions},
 fjournal = {Probability Theory and Related Fields},
 journal = {Probab. Theory Relat. Fields},
 issn = {0178-8051},
 volume = {186},
 number = {3-4},
 pages = {1081--1166},
 year = {2023},
 doi = {10.1007/s00440-022-01166-5},
 zbMATH = {7698591},
 Zbl = {1527.46039}
}

@article{BBgG2009regularization,
 author = {Belinschi, Serban T. and Benaych-Georges, Florent and Guionnet, Alice},
 title = {Regularization by free additive convolution, square and rectangular cases},
 fjournal = {Complex Analysis and Operator Theory},
 journal = {Complex Anal. Oper. Theory},
 issn = {1661-8254},
 volume = {3},
 number = {3},
 pages = {611--660},
 year = {2009},
 doi = {10.1007/s11785-008-0080-z},
 zbMATH = {5655765},
 Zbl = {1187.46055}
}

@article {KT2018limiting,
    AUTHOR = {K\"osters, Holger and Tikhomirov, Alexander},
     TITLE = {Limiting spectral distributions of sums of products of
              non-{H}ermitian random matrices},
      NOTE = {[On table of contents: Vol. 33 (2013)]},
   JOURNAL = {Probab. Math. Statist.},
  FJOURNAL = {Probability and Mathematical Statistics},
    VOLUME = {38},
      YEAR = {2018},
    NUMBER = {2},
     PAGES = {359--384},
      ISSN = {0208-4147,2300-8113},
   MRCLASS = {60B20 (46L54 60E07 60F05)},
  MRNUMBER = {3896715},
       DOI = {10.19195/0208-4147.38.2.6},
       URL = {https://doi-org.kyoto-u.idm.oclc.org/10.19195/0208-4147.38.2.6},
}

@book {NicaSpeicher2006book,
    AUTHOR = {Nica, Alexandru and Speicher, Roland},
     TITLE = {Lectures on the combinatorics of free probability},
    SERIES = {London Mathematical Society Lecture Note Series},
    VOLUME = {335},
 PUBLISHER = {Cambridge University Press, Cambridge},
      YEAR = {2006},
     PAGES = {xvi+417},
      ISBN = {978-0-521-85852-6; 0-521-85852-6},
   MRCLASS = {46L54 (46L53 60C05 81S25)},
  MRNUMBER = {2266879},
MRREVIEWER = {Todd\ Kemp},
       DOI = {10.1017/CBO9780511735127},
       URL = {https://doi-org.ezproxy.lib.uh.edu/10.1017/CBO9780511735127},
}

@misc{B1986lidskii,
 author = {Brown, L. G.},
 title = {Lidskii's theorem in the type {II} case},
 year = {1986},
 howpublished = {Geometric methods in operator algebras, {Proc}. {US}-{Jap}. {Semin}., {Kyoto}/{Jap}. 1983, {Pitman} {Res}. {Notes} {Math}. {Ser}. 123, 1-35 (1986).},
 zbMATH = {4054343},
 Zbl = {0646.46058}
}

@article{BL2001computation,
 author = {Biane, Philippe and Lehner, Franz},
 title = {Computation of some examples of {Brown}'s spectral measure in free probability},
 fjournal = {Colloquium Mathematicum},
 journal = {Colloq. Math.},
 issn = {0010-1354},
 volume = {90},
 number = {2},
 pages = {181--211},
 year = {2001},
 doi = {10.4064/cm90-2-3},
 zbMATH = {1709857},
 Zbl = {0988.22004}
}

@incollection{Z2023brown,
 author = {Zhong, Ping},
 title = {Brown measure of {{\(R\)}}-diagonal operators, revisited},
 booktitle = {Operators, semigroups, algebras and function theory. Volume from IWOTA, Lancaster, UK, virtual, August 16--20, 2021},
 isbn = {978-3-031-38019-8; 978-3-031-38022-8; 978-3-031-38020-4},
 pages = {225--254},
 year = {2023},
 publisher = {Cham: Birkh{\"a}user},
 doi = {10.1007/978-3-031-38020-4_10},
 zbMATH = {7927985}
}

@article{EJ2025density,
 author = {Erd{\H{o}}s, L{\'a}szl{\'o} and Ji, Hong Chang},
 title = {Density of {Brown} measure of free circular {Brownian} motion},
 fjournal = {Documenta Mathematica},
 journal = {Doc. Math.},
 issn = {1431-0635},
 volume = {30},
 number = {2},
 pages = {417--453},
 year = {2025},
 doi = {10.4171/DM/999},
 zbMATH = {8015942},
 Zbl = {1572.46050}
}

@article{BC2016outlier,
 author = {Bordenave, Charles and Capitaine, Mireille},
 title = {Outlier eigenvalues for deformed i.i.d. random matrices},
 fjournal = {Communications on Pure and Applied Mathematics},
 journal = {Commun. Pure Appl. Math.},
 issn = {0010-3640},
 volume = {69},
 number = {11},
 pages = {2131--2194},
 year = {2016},
 doi = {10.1002/cpa.21629},
 zbMATH = {6643251},
 Zbl = {1353.15032}
}

@article{BCC2014spectrum,
 author = {Bordenave, Charles and Caputo, Pietro and Chafa{\"i}, Djalil},
 title = {Spectrum of {Markov} generators on sparse random graphs},
 fjournal = {Communications on Pure and Applied Mathematics},
 journal = {Commun. Pure Appl. Math.},
 issn = {0010-3640},
 volume = {67},
 number = {4},
 pages = {621--669},
 year = {2014},
 doi = {10.1002/cpa.21496},
 zbMATH = {6291318},
 Zbl = {1301.60093}
}

@article{GKZ2011single,
 author = {Guionnet, Alice and Krishnapur, Manjunath and Zeitouni, Ofer},
 title = {The single ring theorem},
 fjournal = {Annals of Mathematics. Second Series},
 journal = {Ann. Math. (2)},
 issn = {0003-486X},
 volume = {174},
 number = {2},
 pages = {1189--1217},
 year = {2011},
 doi = {10.4007/annals.2011.174.2.10},
 zbMATH = {5960725},
 Zbl = {1239.15018}
}

@article{HZ2025deformed,
 author = {Ho, Ching-Wei and Zhong, Ping},
 title = {Deformed single ring theorems},
 fjournal = {Journal of Functional Analysis},
 journal = {J. Funct. Anal.},
 issn = {0022-1236},
 volume = {288},
 number = {5},
 pages = {42},
 note = {Id/No 110797},
 year = {2025},
 doi = {10.1016/j.jfa.2024.110797},
 zbMATH = {7965129},
 Zbl = {1556.60014}
}

@article{B1997free,
 author = {Biane, Philippe},
 title = {On the free convolution with a semi-circular distribution},
 fjournal = {Indiana University Mathematics Journal},
 journal = {Indiana Univ. Math. J.},
 issn = {0022-2518},
 volume = {46},
 number = {3},
 pages = {705--718},
 year = {1997},
 doi = {10.1512/iumj.1997.46.1467},
 url = {www.iumj.indiana.edu/TOC/973.htm},
 zbMATH = {1180062},
 Zbl = {0904.46045}
}

@article{BWZ2023superconvergence,
 author = {Bercovici, Hari and Wang, Jiun-Chau and Zhong, Ping},
 title = {Superconvergence and regularity of densities in free probability},
 fjournal = {Transactions of the American Mathematical Society},
 journal = {Trans. Am. Math. Soc.},
 issn = {0002-9947},
 volume = {376},
 number = {7},
 pages = {4901--4956},
 year = {2023},
 doi = {10.1090/tran/8891},
 zbMATH = {7709423},
 Zbl = {1528.46052}
}

@article{BBS2018eigenvalues,
 author = {Belinschi, Serban T. and {\'S}niady, Piotr and Speicher, Roland},
 title = {Eigenvalues of non-{Hermitian} random matrices and {Brown} measure of non-normal operators: {Hermitian} reduction and linearization method},
 fjournal = {Linear Algebra and its Applications},
 journal = {Linear Algebra Appl.},
 issn = {0024-3795},
 volume = {537},
 pages = {48--83},
 year = {2018},
 doi = {10.1016/j.laa.2017.09.024},
 zbMATH = {6802562},
 Zbl = {1376.15025}
}

@article{BN2008remarkable,
 author = {Belinschi, Serban T. and Nica, Alexandru},
 title = {On a remarkable semigroup of homomorphisms with respect to free multiplicative convolution},
 fjournal = {Indiana University Mathematics Journal},
 journal = {Indiana Univ. Math. J.},
 issn = {0022-2518},
 volume = {57},
 number = {4},
 pages = {1679--1713},
 year = {2008},
 doi = {10.1512/iumj.2008.57.3285},
 zbMATH = {5381919},
 Zbl = {1165.46033}
}

@article{BP1999stable,
 author = {Bercovici, Hari and Pata, Vittorino},
 title = {Stable laws and domains of attraction in free probability theory},
 fjournal = {Annals of Mathematics. Second Series},
 journal = {Ann. Math. (2)},
 issn = {0003-486X},
 volume = {149},
 number = {3},
 pages = {1023--1060},
 year = {1999},
 doi = {10.2307/121080},
 url = {https://eudml.org/doc/120522},
 zbMATH = {1347803},
 Zbl = {0945.46046}
}

@article{Bg2009rectangular,
 author = {Benaych-Georges, Florent},
 title = {Rectangular random matrices, related convolution},
 fjournal = {Probability Theory and Related Fields},
 journal = {Probab. Theory Relat. Fields},
 issn = {0178-8051},
 volume = {144},
 number = {3-4},
 pages = {471--515},
 year = {2009},
 doi = {10.1007/s00440-008-0152-z},
 zbMATH = {5586523},
 Zbl = {1171.15022}
}

@article{Bg2007infinitely,
 author = {Benaych-Georges, Florent},
 title = {Infinitely divisible distributions for rectangular free convolution: classification and matricial interpretation},
 fjournal = {Probability Theory and Related Fields},
 journal = {Probab. Theory Relat. Fields},
 issn = {0178-8051},
 volume = {139},
 number = {1-2},
 pages = {143--189},
 year = {2007},
 doi = {10.1007/s00440-006-0042-1},
 zbMATH = {5180132},
 Zbl = {1129.15019}
}

@article{Bg2010surprising,
 author = {Benaych-Georges, Florent},
 title = {On a surprising relation between the {Marchenko}-{Pastur} law, rectangular and square free convolutions},
 fjournal = {Annales de l'Institut Henri Poincar{\'e}. Probabilit{\'e}s et Statistiques},
 journal = {Ann. Inst. Henri Poincar{\'e}, Probab. Stat.},
 issn = {0246-0203},
 volume = {46},
 number = {3},
 pages = {644--652},
 year = {2010},
 doi = {10.1214/09-AIHP324},
 url = {https://eudml.org/doc/240980},
 zbMATH = {5795078},
 Zbl = {1206.46055}
}

@article{AH2016classical,
 author = {Arizmendi, Octavio and Hasebe, Takahiro},
 title = {Classical scale mixtures of {Boolean} stable laws},
 fjournal = {Transactions of the American Mathematical Society},
 journal = {Trans. Am. Math. Soc.},
 issn = {0002-9947},
 volume = {368},
 number = {7},
 pages = {4873--4905},
 year = {2016},
 doi = {10.1090/tran/6792},
 url = {hdl.handle.net/2115/66350},
 zbMATH = {6551158},
 Zbl = {1351.46063}
}

@incollection{H2021pde,
 author = {Hall, Brian C.},
 title = {{PDE} methods in random matrix theory},
 booktitle = {Harmonic analysis and applications},
 isbn = {978-3-030-61886-5; 978-3-030-61889-6; 978-3-030-61887-2},
 pages = {77--124},
 year = {2021},
 publisher = {Cham: Springer},
 doi = {10.1007/978-3-030-61887-2_5},
 zbMATH = {7444919},
 Zbl = {1493.60014}
}

@article{popakamil2025,
      title={On some properties of free commutators with semicircular variables}, 
      author={Mihai Popa and Kamil Szpojankowski},
      year={2025},
      eprint={2511.13578},
      archivePrefix={arXiv},
      primaryClass={math.OA},
      url={https://arxiv.org/abs/2511.13578}, 
}

@article{lehnerkamil2024integral,
      title={Free Integral Calculus I}, 
      author={Franz Lehner and Kamil Szpojankowski},
      year={2024},
      eprint={2311.04039},
      archivePrefix={arXiv},
      primaryClass={math.OA},
      url={https://arxiv.org/abs/2311.04039}, 
}

@article {Hasebe2016,
    AUTHOR = {Hasebe, Takahiro},
     TITLE = {Free infinite divisibility for powers of random variables},
   JOURNAL = {ALEA Lat. Am. J. Probab. Math. Stat.},
  FJOURNAL = {ALEA. Latin American Journal of Probability and Mathematical
              Statistics},
    VOLUME = {13},
      YEAR = {2016},
    NUMBER = {1},
     PAGES = {309--336},
      ISSN = {1980-0436},
   MRCLASS = {46L54 (60E07)},
  MRNUMBER = {3481440},
       DOI = {10.30757/alea.v13-13},
       URL = {https://doi-org.ezproxy.lib.uh.edu/10.30757/alea.v13-13},
}

@article {LehnerEjsmont2017,
    AUTHOR = {Ejsmont, Wiktor and Lehner, Franz},
     TITLE = {Sample variance in free probability},
   JOURNAL = {J. Funct. Anal.},
  FJOURNAL = {Journal of Functional Analysis},
    VOLUME = {273},
      YEAR = {2017},
    NUMBER = {7},
     PAGES = {2488--2520},
      ISSN = {0022-1236,1096-0783},
   MRCLASS = {46L54 (62E10)},
  MRNUMBER = {3677831},
MRREVIEWER = {B.\ V. Rajarama Bhat},
       DOI = {10.1016/j.jfa.2017.05.007},
       URL = {https://doi-org.ezproxy.lib.uh.edu/10.1016/j.jfa.2017.05.007},
}

@article {LehnerEjsmont2021,
    AUTHOR = {Ejsmont, Wiktor and Lehner, Franz},
     TITLE = {Sums of commutators in free probability},
   JOURNAL = {J. Funct. Anal.},
  FJOURNAL = {Journal of Functional Analysis},
    VOLUME = {280},
      YEAR = {2021},
    NUMBER = {2},
     PAGES = {Paper No. 108791, 29},
      ISSN = {0022-1236,1096-0783},
   MRCLASS = {46L54 (62E10)},
  MRNUMBER = {4159271},
MRREVIEWER = {Wojciech\ M\l otkowski},
       DOI = {10.1016/j.jfa.2020.108791},
       URL = {https://doi-org.ezproxy.lib.uh.edu/10.1016/j.jfa.2020.108791},
}

\end{document}